\documentclass[10pt]{amsart}
\usepackage{fullpage, amsmath, amsthm,amsfonts,url,amssymb,stmaryrd, mathrsfs}
\usepackage{mathrsfs}
\usepackage{amssymb,amscd}
\usepackage{color}
\usepackage{chemarrow}
\usepackage{pictexwd,dcpic}
\usepackage{extarrows}
\usepackage[colorlinks,linkcolor=red,anchorcolor=green,citecolor=blue]{hyperref}
\usepackage{enumitem}

\newtheorem{theorem}{Theorem}[section]
\newtheorem{corollary}[theorem]{Corollary}
\newtheorem{lemma}[theorem]{Lemma}
\newtheorem{proposition}[theorem]{Proposition}

\newtheorem{definition}[theorem]{Definition}

\newtheorem{remark}[theorem]{Remark}

\newtheorem{example}[theorem]{Example}
\newtheorem{conjecture}[theorem]{Conjecture}

\newcommand{\hooklongrightarrow}{\lhook\joinrel\longrightarrow}
\newcommand{\twoheadlongrightarrow}{\relbar\joinrel\twoheadrightarrow}

\newcommand{\us}{\upsilon}
\newcommand{\ra}{\rightarrow}
\newcommand{\lra}{\longrightarrow}

\newcommand{\ul}{\underline}

\newcommand{\bA}{\mathbb A}
\newcommand{\bC}{\mathbb C}

\newcommand{\bG}{\mathbb G}

\newcommand{\Q}{\mathbb Q}
\newcommand{\bR}{\mathbb R}

\newcommand{\Z}{\mathbb Z}

\newcommand{\cN}{\mathcal N}
\newcommand{\cL}{\mathcal L}
\newcommand{\co}{\mathcal O}

\newcommand{\cR}{\mathcal R}
\newcommand{\cH}{\mathcal H}
\newcommand{\cC}{\mathcal C}
\newcommand{\cS}{\mathcal S}
\newcommand{\cD}{\mathcal D}
\newcommand{\cI}{\mathcal I}
\newcommand{\cW}{\mathcal W}

\newcommand{\cM}{\mathcal M}
\newcommand{\cF}{\mathcal F}
\newcommand{\cV}{\mathcal V}
\newcommand{\cE}{\mathcal E}
\newcommand{\cG}{\mathcal G}

\newcommand{\cU}{\mathcal U}

\newcommand{\cP}{\mathcal P}

\newcommand{\fn}{\mathfrak n}
\newcommand{\fh}{\mathfrak h}
\newcommand{\fm}{\mathfrak{m}}
\newcommand{\ub}{\mathfrak b}
\newcommand{\fl}{\mathfrak l}
\newcommand{\fp}{\mathfrak p}

\newcommand{\ug}{\mathfrak g}

\newcommand{\fX}{\mathfrak X}

\newcommand{\fx}{\mathfrak x}

\newcommand{\ft}{\mathfrak t}

\newcommand{\fz}{\mathfrak z}

\newcommand{\fq}{\mathfrak q}

\newcommand{\sW}{\mathscr W}
\newcommand{\sL}{\mathscr L}

\DeclareMathOperator{\la}{\mathrm la}

\DeclareMathOperator{\GL}{\mathrm GL}
\DeclareMathOperator{\gr}{\mathrm gr}

\DeclareMathOperator{\Fil}{\mathrm Fil}

\DeclareMathOperator{\Gal}{\mathrm Gal}
\DeclareMathOperator{\Hom}{\mathrm Hom}

\DeclareMathOperator{\cris}{\mathrm cris}
\DeclareMathOperator{\rig}{\mathrm rig}
\DeclareMathOperator{\an}{\mathrm an}
\DeclareMathOperator{\Spec}{\mathrm Spec}

\DeclareMathOperator{\Frob}{\mathrm Frob}
\DeclareMathOperator{\Ind}{\mathrm Ind}
\DeclareMathOperator{\unr}{\mathrm unr}

\DeclareMathOperator{\Ker}{\mathrm Ker}

\DeclareMathOperator{\Spm}{\mathrm Spm}

\DeclareMathOperator{\Ima}{\mathrm Im}
\DeclareMathOperator{\SL}{\mathrm SL}
\DeclareMathOperator{\lalg}{\mathrm lalg}
\DeclareMathOperator{\cl}{\mathrm cl}

\DeclareMathOperator{\id}{\mathrm id}

\DeclareMathOperator{\dett}{\mathrm det}
\DeclareMathOperator{\alg}{\mathrm alg}

\DeclareMathOperator{\soc}{\mathrm soc}

\DeclareMathOperator{\Supp}{\mathrm Supp}

\DeclareMathOperator{\red}{\mathrm red}

\DeclareMathOperator{\wt}{\mathrm wt}
\DeclareMathOperator{\cfs}{\mathrm fs}
\DeclareMathOperator{\Norm}{\mathrm Norm}
\begin{document}\setcounter{tocdepth}{2}
\setcounter{secnumdepth}{3}
\title{Some results on locally analytic socle for $\GL_n(\Q_p)$}
\author{Yiwen Ding}
\address{Department of Mathematics, Imperial College London}
\email{y.ding@imperial.ac.uk}
\thanks{The author is supported by EPSRC grant EP/L025485/1.}
\maketitle
\begin{abstract}
     We study some closed rigid subspaces of the eigenvarieties, constructed by using the Jacquet-Emerton functor for parabolic non-Borel subgroups. As an application (and motivation), we prove some new results on Breuil's locally analytic socle conjecture for  $\GL_n(\Q_p)$.
\end{abstract}
\tableofcontents
\addtocontents{toc}{\protect\setcounter{tocdepth}{1}}

\section{Introduction}
 This note is devoted to prove some new results on Breuil's locally analytic socle conjecture for $\GL_n(\Q_p)$. We  recall the conjecture, summarize some results and sketch the proof  in $\GL_3(\Q_p)$ case  in the introduction.

Let $F$ be a quadratic  imaginary extension of $\Q$ with $p$ split in $F$, and we fix a place $u$ above $p$; let $G$ be a definite unitary group over $\Q$ associated to $F/\Q$ which is split at $p$. Let $E$ be a finite extension of $\Q_p$ sufficiently large, $U^p$ a compact open subgroup of $G(\bA^{\infty,p})$, put
\begin{equation*}
  \widehat{S}(U^p,E):=\{f: G(\Q)\backslash G(\bA^{\infty})/U^p \ra E\ |\ \text{$f$ is continuous}\},
\end{equation*}
which is a Banach space over $E$ equipped with a continuous action of $G(\Q_p)\cong \GL_3(\Q_p)$ (this isomorphism depends on the choice of $u$), and a continuous action of (commutative) Hecke algebra $\cH^p$ outside $p$. The action of $\cH^p$ commutes with that of $\GL_3(\Q_p)$. Let $\rho$ be a continuous representation of $\Gal(\overline{F}/F)$ over $E$ associated to automorphic forms of $G$, and we associate to $\rho$ a maximal ideal of $\cH^p$ (shrinking $\cH^p$ if needed). Suppose $(\widehat{S}(U^p,E)_{\lalg})^{\fm_\rho}\neq 0$, where ``$\lalg$" denotes the locally algebraic vectors for $\GL_3(\Q_p)$, $(\cdot)^{\fm_{\rho}}$ denotes the maximal $E$-vector space on which $\cH^p$ acts via $\cH^p \twoheadrightarrow \cH^p/\fm_{\rho}$. Put
\begin{equation*}
  \widehat{\Pi}(\rho):=\widehat{S}(U^p,E)^{\fm_{\rho}},
\end{equation*}
which is an admissible unitary Banach representation of $\GL_3(\Q_p)$, and is supposed to be (a direct sum) of the right representation corresponding to $\rho_p:=\rho|_{\Gal(\overline{F_u}/F_u)\cong \Gal(\overline{\Q_p}/\Q_p)}$ in $p$-adic Langlands programme (\cite{Br0}). The structure of $\widehat{\Pi}(\rho)$ is still quite mysterious. In \cite{Br13I}, Breuil made a conjecture on the locally analytic socle of $\widehat{\Pi}(\rho)$, which we recall now.

Suppose $\rho_p$ is crystalline and very regular (cf. Def.\ref{def: gln-fon}). Let $\ul{h}:=(h_1,h_2,h_3)\in \Z^{\oplus}$ (with $h_1<h_2<h_3$) be the Hodge-Tate weights of $\rho_p$, $\ul{\phi}:=(\phi_1,\phi_2,\phi_3)\in E^3$ be the (ordered) eigenvalues of the crystalline Frobenius $\varphi$ on $D_{\cris}(\rho_p)$. For any $w\in S_3$, put $w(\ul{\phi}):=(\phi_{w^{-1}(1)}, \phi_{w^{-1}(2)}, \phi_{w^{-1}(3)})$, which is called a refinement for $\rho_p$. The local Langlands correspondance thus associates to $\rho_p$ a locally algebraic representation of $\GL_3(\Q_p)$ over $E$:
\begin{equation*}
  C(1,w):=\big(\Ind_{\overline{B}(\Q_p)}^{\GL_3(\Q_p)} \psi_w \delta_B^{-1} \big)^\infty \otimes_E \sL(\lambda)
\end{equation*}
where $\lambda:=(-h_1,1-h_2,2-h_3)$ is a dominant weight for $T(\Q_p)$ \big(the group of diagonal matrices of $\GL_3(\Q_p)$\big), $\sL(\lambda)$ denotes the irreducible algebraic representation of $\GL_3(\Q_p)$ with highest weight $\lambda$, $\psi_w:=\unr(\phi_{w^{-1}(1)} p^{-2}) \otimes \unr(\phi_{w^{-1}(2)} p^{-1}) \otimes \unr(\phi_{w^{-1}(3)})$, and $\delta_B=\unr(p^{-2})\otimes 1 \otimes \unr(p^2)$.  Since $\rho_p$ is very regular, the representation $C(1,w)$ are all isomorphic and irreducible. For each refinement $w(\ul{\phi})$ of $\rho_p$, one can get a triangulation of $D_{\rig}(\rho_p)$ of parameter
\begin{multline*}
  \delta=(\delta_1, \delta_2, \delta_3)\\
  =\Big(\unr(\phi_{w^{-1}(1)})x^{-h_{w^{\alg}(w)^{-1}(1)}}, \unr(\phi_{w^{-1}(2)})x^{-h_{w^{\alg}(w)^{-1}(2)}}, \unr(\phi_{w^{-1}(3)})x^{-h_{w^{\alg}(w)^{-1}(3)}}\Big)
\end{multline*}
with $w^{\alg}(w)\in S_3$ (determined by $w$ and $\rho_p$).  Recall the refinement is called \emph{non-critical}, if $w^{\alg}(w)=1$. For each pair $(w^{\alg},w)\in S_3 \times S_3$, Breuil defined an irreducible locally analytic representation $C(w^{\alg},w)$ (cf. (\ref{equ: gln-ngaa})), and conjectured
\begin{conjecture}[$\text{\cite{Br13I}, \cite[Conj.5.3]{Br13II}}$]\label{conj: gln-cwn}
  For $(w^{\alg},w)\in S_3\times S_3$, $C(w^{\alg},w)$ is a subrepresentation of $\widehat{\Pi}(\rho)$ if and only if $w^{\alg}\leq w^{\alg}(w)$ for the Bruhat's ordering.
\end{conjecture}Roughly speaking, $\soc\widehat{\Pi}(\rho)$ would (conjecturally) measure the \emph{criticalness} of $\rho_p$.
In \cite{Br13II}, Breuil proved (some results were also obtained in \cite{BCho14})
\begin{theorem}[cf. $\text{\cite[Thm.1.2]{Br13II}}$]\label{thm: gln-sio}
  (1) If $C(w^{\alg},w)$ is a subrepresentation of $\widehat{\Pi}(\rho)$, then $w^{\alg}\leq w^{\alg}(w)$.

  (2) If $w^{\alg}(w)\neq 1$, then there exists $w^{\alg}\neq 1$, such that $C(w^{\alg},w)$ is a subrepresentation of $\widehat{\Pi}(\rho)$.
\end{theorem}
In particular, when $\lg w^{\alg}(w)\leq 1$, the conjecture \ref{conj: gln-cwn} was proved. In the $\GL_n$ case, one needs to put more global hypothesis to get Thm.\ref{thm: gln-sio} (2), and in general one only gets a weaker version of Thm.\ref{thm: gln-sio} (1) (cf. \cite[Thm.1.2]{Br13II}). The main result of this note (in $\GL_3(\Q_p)$ case) is the following theorem which improves Thm.\ref{thm: gln-sio} (2).
\begin{theorem}[cf. Thm.\ref{thm: gln-3dG}, Cor.\ref{cor: gln-ang}]\label{thm: gln-wsn}
  Let $s\in S_3$ be a simple reflection (i.e. $s\in \Delta:=\{(12), (23)\}$), then $C(s,w)$ is a subrepresentation of $\widehat{\Pi}(\rho)$ if and only if $s\leq w^{\alg}(w)$. In particular, if $\lg w^{\alg}(w)\geq 2$, then $\oplus_{s\in \Delta} C(s,w)$ is a subrepresentation of $\widehat{\Pi}(\rho)$.
\end{theorem}
Let's remark that the general $\GL_n(\Q_p)$ case is more subtle (beside the global hypothesis), but we do prove that if $\lg w^{\alg}(w)\geq 2$, there exist more than one $w^{\alg}$ such that $C(w^{\alg},w)\hookrightarrow \widehat{\Pi}(\rho)$ (see Cor.\ref{cor: gln-nmf}).

The proof of Thm.\ref{thm: gln-wsn} follows the same strategy of \cite{Br13II}, i.e. using results on the geometry of the eigenvariety (due to Bergdall \cite{Bergd14}, Chenevier \cite{Che}) and locally analytic representation theory (adjunction formulas due to Breuil \cite{Br13II})  to prove the existence of companion points (on the eigenvariety), which would correspond to irreducible components of $\soc \widehat{\Pi}(\rho)$. While, a key idea in this note, inspired by the adjunction formula \cite[Thm.4.3]{Br13II} (see also \cite[Rem.9.11 (ii)]{Br13II}),  is to \emph{locate} some companion points by considering closed subspaces of the eigenvariety constructed via the Jacquet-Emerton functor for parabolic non-Borel subgroups. In fact, such closed subspaces were already constructed by Hill  and Loeffler \cite{HL} (although their motivation was rather different from ours).

We sketch the proof of Thm.\ref{thm: gln-wsn} for $s=(23)$ and discuss some intermediate results. We keep the notations.
\subsection*{Jacquet-Emerton functors}Let $P\supset B$ with the Levi subgroup $L_P=\GL_2\times \GL_1$ (the case $s=(12)$ would use the other maximal parabolic proper subgroup). Denote by $\sL_P(\lambda)$ the irreducible algebraic representation of $L_P$ with highest weight $\lambda$, for an admissible locally analytic representation $V$ of $\GL_3(\Q_p)$, we put (cf. \cite{HL}, and \S \ref{thm: gln-wsn})
\begin{equation*}
 J_{B,(P,\lambda)}(V):= J_{B\cap L_P}\big(\big(J_P(V)\otimes_E \sL_P(\lambda)'\big)_{\infty} \otimes_E \sL_P(\lambda)\big)
\end{equation*}
where ``$(\cdot)_{\infty}$" denotes the smooth vectors for $\GL_2(\Q_p)$ which acts on $J_P(V)\otimes_E \sL_P(\lambda)'$ via $\GL_2 \hookrightarrow \GL_2\times \GL_1\cong L_P$, $\sL_P(\lambda)'$ denotes the algebraic dual of $\sL_P(\lambda)$. In fact, $J_{B,(P,\lambda)}(V)$ is a closed subrepresentation  (of $T(\Q_p)$) of the usual Jacquet-Emerton module $J_B(V)$, and thus is an essentially admissible locally analytic representation of $T(\Q_p)$. We would use the subfunctor $J_{B,(P,\lambda)}(\cdot)$ of $J_{B}(\cdot)$ to construct a closed subspace of the eigenvariety. The adjunction property for $J_{B,(P,\lambda)}(\cdot)$, which we discuss below, would allow us to get some nice properties of such closed subspace (cf. Thm.\ref{thm: gln-aen}, \ref{thm: gln-tnn}).

\subsection*{Adjunction formulas}Suppose $V$ is moreover very strongly admissible  (cf. \cite[Def.0.12]{Em2}), we have an adjunction formula (cf. Thm.\ref{thm: gln-ycn}) (obtained by combining Breuil's adjunction formula \cite[Thm.4.3]{Br13II} and the adjunction formula for the classical Jacquet functor):
  \begin{multline}\label{equ: gln-pcfg}
    \Hom_{\GL_3(\Q_p)}\Big(\cF_{\overline{P}}^{G}\Big( \big(\text{U}(\ug) \otimes_{\text{U}(\overline{\fp})}\sL_P(\lambda)'\big)^{\vee}, \big(\Ind_{\overline{B}(\Q_p)\cap L_{P}(\Q_p)}^{L_{P}(\Q_p)} \psi \otimes_E \delta_{B}^{-1}\big)^{\infty}\Big), V\Big)
    \\
    \xlongrightarrow{\sim} \Hom_{T(\Q_p)}\big( \psi\otimes_E \chi_{\lambda}, J_{B,(P,\lambda)}(V)\big),
  \end{multline}
   where $\psi$ is a finite length smooth representation of $T(\Q_p)$ over $E$,  $\chi_{\lambda}$ denotes the algebraic character of $T(\Q_p)$ with weight $\lambda$ and we refer to \S \ref{def: gln-ryr} for the representations $\cF_{\overline{P}}^G(\cdot,\cdot)$ etc.; meanwhile, recall that for $J_B(V)$, by \cite[Thm.4.3]{Br13II}, one has
  \begin{equation}\label{equ: gon-ebgg}
    \Hom_{\GL_3(\Q_p)}\Big(\cF_{\overline{B}}^{G}\Big( \big(\text{U}(\ug) \otimes_{\text{U}(\overline{\ub})}(-\lambda)\big)^{\vee}, \psi\otimes_E \delta_B^{-1}\Big), V\Big)
    \xlongrightarrow{\sim} \Hom_{T(\Q_p)}\big( \psi\otimes_E \chi_{\lambda}, J_B(V)\big).
  \end{equation}
 Note the generalized Verma module $\text{U}(\ug) \otimes_{\text{U}(\overline{\fp})}\sL_P(\lambda)'$ has $2$ irreducible components, while the Verma module $\text{U}(\ug) \otimes_{\text{U}(\overline{\ub})}(-\lambda)$ has $6$ irreducible components \big(with $\text{U}(\ug) \otimes_{\text{U}(\overline{\fp})}\sL_P(\lambda)'$ as a quotient\big), consequently the locally analytic representation (in (\ref{equ: gln-pcfg}))
 \begin{equation}\label{equ: gln-fPgg}\cF_{\overline{P}}^{G}\Big( \big(\text{U}(\ug) \otimes_{\text{U}(\overline{\fp})}\sL_P(\lambda)'\big)^{\vee}, \big(\Ind_{\overline{B}(\Q_p)\cap L_{P}(\Q_p)}^{L_{P}(\Q_p)} \psi \otimes_E \delta_{B}^{-1}\big)^{\infty}\Big),\end{equation}
 as a quotient of $\cF_{\overline{B}}^{G}\Big( \big(\text{U}(\ug) \otimes_{\text{U}(\overline{\ub})}(-\lambda)\big)^{\vee}, \psi\otimes_E \delta_B^{-1}\Big)$, has much fewer irreducible components. For example,  when $\psi=\psi_{w}$, then the representation in (\ref{equ: gln-fPgg}) has the form (where the line denotes an extension)
 \begin{equation}\label{equ: gln-w1t}
   C(s,w) \ \rule[3pt]{7mm}{.4pt}\  C(1,w);
 \end{equation}
while,  $\cF_{\overline{B}}^{G}\Big( \big(\text{U}(\ug) \otimes_{\text{U}(\overline{\ub})}(-\lambda)\big)^{\vee}, \psi_w\delta_B^{-1}\Big)$ has the form
 \begin{equation}\label{equ: gln-0ws}
  \begindc{\commdiag}[40]
    \obj(2,0)[a]{$C(ss',w)$}
    \obj(4,0)[b]{$C(s',w)$}
    \obj(0,1)[c]{$C(s'ss',w)$}
    \obj(6,1)[d]{$C(1,w)$}
    \obj(2,2)[e]{$C(s's,w)$}
    \obj(4,2)[f]{$C(s,w)$}
    \mor{a}{b}{}[+1,2]
    \mor{a}{f}{}[+1,2]
     \mor{a}{c}{}[+1,2]
     \mor{b}{d}{}[+1,2]
     \mor{b}{e}{}[+1,2]
     \mor{d}{f}{}[+1,2]
     \mor{f}{e}{}[+1,2]
     \mor{e}{c}{}[+1,2]
  \enddc
  \end{equation}
  where $s'$ denotes the simple reflection different from $s$. The adjunction property (\ref{equ: gln-pcfg}) is somehow the key point of this note.

\subsection*{Eigenvariety and closed subspaces} Consider $J_{B,(P,\lambda)}\big(\widehat{S}(K^p,U)_{\an}\big)$, which is an essentially admissible locally analytic representation of $T(\Q_p)$ equipped moreover with a continuous action of $\cH^p$. Following Emerton, one gets a rigid space $\cV(P,\lambda)$ over $E$ such that there exists a bijection
\begin{equation*}
  \cV(P,\lambda)(\overline{E}) \xlongrightarrow{\sim} \Big\{(\chi, \fh)\in \widehat{T}(\overline{E}) \times \Spec \cH^p(\overline{E})\ \big|\ \big(J_{B,(P,\lambda)}\big(\widehat{S}(K^p,U)_{\an}\big)\otimes_E \overline{E}\big)^{T(\Q_p)=\chi,\cH^p=\fh}\neq 0\Big\}.
\end{equation*}
Indeed, $\cV(P,\lambda)$ is a closed subspace of the general eigenvariety $\cV$ \big(constructed from $J_B\big(\widehat{S}(U^p,E)_{\an}\big)$\big). By the definition of $J_{B,(P,\lambda)}(\cdot)$, one easily sees $(\chi_1\otimes \chi_2 \otimes \chi_3,\fh)\in \cV(P,\lambda)(\overline{E})$ implies that $\wt(\chi_1)=-h_1$, $\wt(\chi_2)=1-h_2$. However, the rigid space $\cV(P,\lambda)$ would be more subtle than the closed subspace, denoted by $\cV(\lambda)$, of $\cV$ lying above the corresponding weight space.
One has
\begin{theorem}[cf. Thm.\ref{thm: gln-vv0i}]\label{thm: gln-aen}
  The classical points are Zariski-dense in $\cV(P,\lambda)$.
\end{theorem}
Thus one can view (the reduced subspace of) $\cV(P,\lambda)$ as the Zariski closure of the classical points in $\cV(\lambda)$. To prove Thm.\ref{thm: gln-aen}, one needs a classicality criterion for closed points in $\cV(P,\lambda)$  proved in  \S \ref{sec: 3.2.1} using (\ref{equ: gln-pcfg}), which is stronger than the general classicality criterion for the points in $\cV$ (e.g. see \cite[\S 4.7.3]{Che}). Indeed, the general classicality criterion for points in $\cV$ seems \emph{not} enough for the density of classical points in $\cV(P,\lambda)$ since certain weights for  $\cV(P,\lambda)$ are fixed.

\subsection*{Some \'etaleness results}Denote by $\fh_{\rho}: \cH^p\ra \cH^p/\fm_{\rho}\ra E$, one can associate to $(\rho,w)$ a classical point $z_{\rho,w}:=(\chi_{w}, \fh_{\rho})$ of $\cV(P,\lambda)$, where $\chi_w=\psi_w \chi_{\lambda}$. Consider the composition $\kappa: \cV(P,\lambda)\ra \widehat{T}\xrightarrow{p_3} \widehat{\Q_p^{\times}} \ra \widehat{\Z_p^{\times}}$, where $p_3$ denotes the projection to the third factor. One has (compare the second statement with \cite[Thm.4.8]{Che11})
\begin{theorem}[cf. Prop.\ref{prop: gln-hlt}, Thm.\ref{thm: gln-tv0} and the proof of Thm.\ref{thm: gln-3dG}]\label{thm: gln-tnn}
  If $C(s,w)$ is not a subrepresentation of $\widehat{\Pi}(\rho)_{\an}$, then $\kappa$ is \'etale at  $z_{\rho,w}$. Consequently (by Thm.\ref{thm: gln-sio} (1)), if $w^{\alg}(w)\in \{1,(12)\}$ (which is in fact the Weyl group of $L_P$), then $\kappa$ is \'etale at $z_{\rho,w}$.
\end{theorem}To prove this theorem, as in \cite{Br13II} (e.g. see the proof of \cite[Thm.9.10]{Br13II}), we use the adjunction formula (\ref{equ: gln-pcfg}) 
and the same arguments in the proof of \cite[Thm.4.8]{Che11}.
Indeed, applying the adjunction formula (\ref{equ: gln-pcfg})
to  the generalized eigenspace $J_{B,(P,\lambda)}(\widehat{S}(U^p,E)_{\an})[T(\Q_p)=\chi_w,  \cH^p=\fh_{\rho}]$
then using (\ref{equ: gln-w1t}),
it's not difficult to  show that if $C(s,w)$
is not a subrepresentation of $\widehat{\Pi}(\rho)$, then
\begin{equation*}
  J_{B,(P,\lambda)}(\widehat{S}(U^p,E)_{\an})[T(\Q_p)=\chi_w, \cH^p=\fh_{\rho}]=J_{B,(P,\lambda)}(\widehat{S}(U^p,E)_{\lalg})[T(\Q_p)=\chi_w, \cH^p=\fh_{\rho}]
\end{equation*}
(such equality might be viewed as an infinitesimal classicality result), from which one can deduce the \'etaleness result for $\cV(\lambda,P)$
by the same argument as in the proof of \cite[Thm.4.8]{Che11}.

Conversely, we have the following theorem due to Bergdall (cf. \cite[Thm.B]{Bergd14}).
\begin{theorem}[cf. Thm.\ref{thm: gln-tvI0}] \label{thm: gln-wnn}If $w^{\alg}(w)\notin \{1,(12)\}$, then $\kappa$ is not \'etale at $z_{\rho,w}$.
  \end{theorem}

  Combing Thm.\ref{thm: gln-tnn} and Thm.\ref{thm: gln-wnn}, Thm.\ref{thm: gln-wsn} thus follows. We refer to the body of the text for more detailed and more precise statements (with slightly different notations).
%
%

\subsection*{Acknowledgement} The debt that this work owes  to \cite{Br13II} is clear, and I also would like to thank Christophe Breuil for answering my questions.
\addtocontents{toc}{\protect\setcounter{tocdepth}{2}}
\section{Locally analytic representations and Jacquet-Emerton functors}Let $G$ be a split connected reductive algebraic group over $\Q_p$, $T$ be a split maximal torus of $G$, $Z_G\subseteq T$ the center of $G$,  $B$ a Borel subgroup of $G$ containing $T$,  $P\supset B$ a parabolic subgroup of $G$ with $L_P$ the Levi subgroup and $N_P$ the nilpotent radical, thus $T\cong L_B$ and put $N:=N_B$. Let $\ug$, $\ft$, $\fz_G$, $\ub$, $\fp$, $\fl_P$, $\fn_P$, $\fn$ denote the associated Lie algebras over $\Q_p$ of $G$, $T$, $Z_G$, $B$, $P$, $L_P$, $N_P$, $N$ respectively. Let $E$ be a finite extension of $\Q_p$, with $\co_E$ the ring of integers and $\varpi_E$ a uniformizer of $\co_E$.
\subsection{The BGG category $\co^{\fp}$ and the representations $\cF_{P}(M,\pi)$}
\begin{definition}[$\text{\cite[\S 9.3]{Hum08}}$]\label{def: gln-ryr}
  Let $\co^{\fp}$ be the full subcategory of the category of linear representations of $\ug$ on $E$-vector spaces made out of representations $M$ such that:
  \begin{enumerate}
    \item $M$ is a finite type $\text{U}(\ug)\otimes_{\Q_p} E$-module;
    \item $M|_{\text{U}(\fl_P)}$ is a direct sum of irreducible algebraic $\text{U}(\fl_P)\otimes_{\Q_p} E$-modules;
    \item for all $v\in M$, the $E$-vector space $\big(U(\fn_P)\otimes_{\Q_p} E\big)v$ is finite dimensional.
  \end{enumerate}
\end{definition}
One has (cf. \cite[\S 1.1, \S 1.2, \S 9.3]{Hum08})
\begin{theorem}
  (1) The category $\co^{\fp}$ is abelian, closed under submodules, quotients and finite direct sums.

  (2) For $P_1 \subset P_2$ two parabolic subgroups of $G$, $\co^{\fp_2}$ is a full subcategory of $\co^{\fp_1}$.

  (3) Let $W$ be an irreducible algebraic representation of $L_P$, the \emph{generalized Verma module} $\text{U}(\ug) \otimes_{\text{U}(\fp)} W$ lies in $\co^{\fp}$, and admits a unique irreducible quotient denoted by $M(\lambda)$ where $\lambda$ is the highest weight of $W$.
\end{theorem}
Denote by $\Phi$ the root system of $G$, $\Delta$ the set of simple roots with respect to $B$ and $\Phi^+$ (resp. $\Phi^-$) the set  of positive (resp. negative) roots. Any parabolic subgroup $P$ containing $B$ corresponds thus to a set of positive simple roots denoted by $\Delta_P$, which is the simple roots of $L_P$ with respect to $B\cap L_P$ (thus $\Delta_B=\emptyset$). Consider the $E$-vector space $\ft^*:=\Hom_{\Q_p}(\ft, E)$. Any element in $\ft^*$ is called a weight of $\ft$ (or a weight for $G$). Any root $\alpha\in \Phi$ can be viewed as a weight still denoted by $\alpha$, and one has a natural embedding $\oplus_{\alpha \in \Delta} E\alpha\hookrightarrow \ft^*$. In fact, one has $\cap_{\alpha\in \Phi} \Ker(\alpha)=\fz_G$, thus a natural decomposition $\ft^*\cong \big(\oplus_{\alpha\in \Delta} E \alpha\big) \oplus \fz_G^*$. Recall the root space $\oplus_{\alpha\in \Delta} E\alpha$ is equipped with a natural inner product (using Killing form) which extends to $\ft^*$ by putting $\langle a,z\rangle=0$ for all $z\in \fz_G^*$.  Let $\alpha^{\vee}:=2\alpha/\langle \alpha,\alpha\rangle$. We call a weight $\lambda$ is \emph{$P$-dominant} if $\langle \lambda, \alpha^{\vee}\rangle\in \Z_{\geq 0}$ for all $\alpha \in \Delta_P$, $\lambda$ is called \emph{dominant} if $\lambda$ is $G$-dominant.

Let $\lambda \in \ft^*$, denote by $M(\lambda):=\text{U}(\ug) \otimes_{\text{U}(\ub)} \lambda$ the Verma module with highest weight $\lambda$, denote by $\sL(\lambda)$ the unique simple quotient of $M(\lambda)$. By \cite[Thm.1.3]{Hum08}, every simple object in $\co^{\ub}$ is isomorphic to a such $\sL(\lambda)$. Note if $\lambda$ is dominant, $\sL(\lambda)$ is thus the unique finite dimensional algebraic representation of $G$ with highest weight $\lambda$.
\begin{proposition}[$\text{cf. \cite[p.185]{Hum08}}$]
   Let $\lambda\in \ft^*$, $\sL(\lambda)\in \co^{\ub}$ lies in $\co^{\fp}$ if and only if $\lambda$ is $P$-dominant. In particular, there exists a unique maximal parabolic $P$ of $G$ containing $B$ such that $\sL(\lambda)\in \co^{\fp}$.
\end{proposition}

Following Orlik-Strauch \cite{OS}, one can associate a locally analytic representation $\cF_P^G(M,\pi)$ of $G(\Q_p)$ with $M\in \co^{\fp}$ and $\pi$ being a finite length smooth admissible representation of $L_P(\Q_p)$. We recall the constructions and some properties.

 For a Hausdorff locally convex topological vector space $V$ over $E$, denote by $\cC^{\la}(G(\Q_p),V)$ the $E$-vector space of locally analytic functions of $G(\Q_p)$ with values in $V$, equipped with the finest locally convex topology. This space is equipped with a right regular locally analytic $G(\Q_p)$-action $g(f)(g')=f(g'g)$. While this space can also be equipped with another locally analytic $G(\Q_p)$-action given by $(g\cdot f)(g')=f(g^{-1}g')$, which induces by derivation a continuous $\ug$-action
 \begin{equation}\label{equ:gln-g0=}
   (\fx\cdot f)(g)=\frac{d}{dt} f(\exp(-t\fx)g)\big|_{t=0}
 \end{equation}
 for $\fx\in \ug$, $f\in \cC^{\la}(G(\Q_p),V)$ and $g\in G(\Q_p)$.

 Let $M\in \co^{\fp}$, and $W\subseteq M$ be a finite dimensional algebraic representation of $\fp$ over $E$, by Def.\ref{def: gln-ryr}, enlarging $W$, we can (and do) assume $W$ generates $M$ over $\text{U}(\ug)\otimes_{\Q_p} E$. One has thus an exact sequence in the category $\co^{\fp}$
 \begin{equation*}
   0 \ra \Ker(\phi) \ra \text{U}(\ug)\otimes_{\text{U}(\fp)} W \xrightarrow{\phi} M \ra 0.
 \end{equation*}The $\fp$-action on $W$ can be lifted to a unique algebraic $P(\Q_p)$-action (cf. \cite[Lem.3.2]{OS}). Denote by $W':=\Hom_E(W,E)$ with $(p(f))(v):=f(p^{-1}v)$ for $p\in P(\Q_p)$ and $v\in W$.

Let $\pi$ be a finite length smooth admissible representation of $L_P(\Q_p)$ over $E$, each element of $\text{U}(\ug)\otimes_{\Q_p} W$ would induce a continuous morphism
\begin{equation*}
  \cC^{\la}\big(G(\Q_p), W'\otimes_E \pi\big) \lra \cC^{\la}\big(G(\Q_p), \pi\big)
\end{equation*}
with $(\fx\otimes v)(f):=[g\mapsto (\fx\cdot f)(g)(v)]$, where $\fx\in \text{U}(\ug)$, $v\in W'$, $f\in  \cC^{\la}\big(G(\Q_p),W'\otimes_E \pi\big)$, $\fx\cdot f$ denotes the $\ug$-action on $\cC^{\la}\big(G(\Q_p),W'\otimes_E \pi\big)$ as in (\ref{equ:gln-g0=}). Consider $\big(\Ind_{P(\Q_p)}^{G(\Q_p)}W'\otimes_E \pi\big)^{\an}\subseteq \cC^{\la}\big(G(\Q_p), W'\otimes_E \pi\big)$, by restriction, one gets thus an $E$-linear map
\begin{equation}\label{equ: gln-wae}
  \text{U}(\ug)\otimes_E W \lra \Hom_{E}\Big(\big(\Ind_{P(\Q_p)}^{G(\Q_p)}W'\otimes_E \pi\big)^{\an}, \cC^{\la}\big(G(\Q_p), \pi\big)\Big).
\end{equation}
One can check this map factors through $\text{U}(\ug)\otimes_{\text{U}(\fp)} W$ (cf. \cite[Lem.2.1]{Br13I}).
Following \cite[\S 4]{OS}, put (where  $X\cdot f$ is given via (\ref{equ: gln-wae}))
\begin{equation*}
  \cF_P^G(M,\pi):=\Big\{f\in \big(\Ind_{P(\Q_p)}^{G(\Q_p)}W'\otimes_E \pi\big)^{\an}, \ X\cdot f=0,\ \forall X\in \Ker(\phi)\Big\},
\end{equation*}
which can be checked to be  a closed subrepresentation of $\big(\Ind_{P(\Q_p)}^{G(\Q_p)}W'\otimes_E \pi\big)^{\an}$ and independent of the choice of $W$ (thus only depends on $M$ and $\pi$, cf. \cite[Prop.4.5]{OS}).
\begin{theorem}[$\text{\cite[Thm]{OS}}$]\label{thm: gln-pst}Keep the above notations.

  (1) $\cF_P^G(M,\pi)$ is non-zero if and only if $M$ and $\pi$ are non-zero.

  (2) The functor $(M,\pi)\mapsto \cF_P^G(M,\pi)$ (covariant on $\pi$, and contravariant  on $M$) is exact in both arguments.

  (3) Let $Q\supset P$ be another parabolic subgroup and assume $M$ lies in $\co^{\fq}\subset \co^{\fp}$, then
  \begin{equation*}\cF_P^G(M,\pi) \cong \cF_Q^G\big(M,\big(\Ind_{P(\Q_p)\cap L_Q(\Q_p)}^{L_Q(\Q_p)} \pi\big)^{\infty}\big),\end{equation*}
  where $(\cdot)^{\infty}$ denotes the usual smooth parabolic induction.

  (4) Suppose the following hypothesis
  \begin{itemize}
    \item if $\Phi$ has irreducible components of type $B$, $C$ or $F_4$, then $p$ is odd,
    \item if $\Phi$ has irreducible components of type $G_2$, then $p>3$;
  \end{itemize}if $M$ and $\pi$ are irreducible and $P$ is the maximal parabolic subgroup of $M$, then $\cF_P^G(M,\pi)$ is an irreducible representation of $G(\Q_p)$.
\end{theorem}
\begin{remark}
  Keep the notations of Thm.\ref{thm: gln-pst} and the hypothesis in (4), if the representation $\big(\Ind_{P(\Q_p)}^{G(\Q_p)}\pi\big)^{\infty}$ is irreducible, we see the irreducible components of $\cF_P^G(M,\pi)$ are exactly $\cF_P^G(M_i,\pi)$ where the $M_i$'s are the irreducible components of $M$ in $\co^{\fp}$.
\end{remark}
\subsection{Jacquet-Emerton functors}
Let $V$ be an essentially admissible locally analytic representation of $G(\Q_p)$ over $E$ (cf. \cite[Def.6.4.9]{Em04}). Let $N_P^o$ be an open compact subgroup of $N_P(\Q_p)$,
\begin{equation}\label{equ: gln-1np}
  L_P(\Q_p)^+:=\big\{g \in L_P(\Q_p)\ |\ g N_P^o g^{-1} \subseteq N_P^o\big\}.
\end{equation}
The closed subspace $V^{N_P^o}$ \big(of vectors fixed by $N_P^o$\big) is equipped with a natural $L_P(\Q_p)^+$-action given by (cf. \cite[\S 3.4]{Em11})
\begin{equation}\label{equ: gln-gnp}
  \pi_g(v):=\frac{1}{|N_P^o/gN_P^og^{-1}|} \sum_{n\in N_P^o/gN_P^og^{-1}} (ng) v, \ v\in V^{N_P^o}, \ g\in L_P(\Q_p)^+.
\end{equation}
Following Emerton (\cite[Def.3.2.1, 3.4.5]{Em11}), put
\begin{equation*}J_P(V):=(V^{N_P^o})_{\cfs}:=\cL_{Z_{L_P}(\Q_p)^+}\Big(\co\big(\widehat{Z_{L_P}}\big), V^{N_P^o}\Big)
 \end{equation*}where  ``$\cL$" signifies continuous linear maps, ``$\cfs$" signifies \emph{finite slope}, $Z_H$ denotes the center of $H$ for an algebraic group $H$, $Z_{L_P}(\Q_p)^+:=Z_{L_P}(\Q_p)\cap L_P(\Q_p)^+$, $\widehat{Z}$ denotes the rigid space over $E$ parameterizing locally analytic characters of $Z(\Q_p)$ for a commutative algebraic group $Z$ over $\Q_p$ and $\co\big(\widehat{Z}\big)$ denotes the global sections of $\widehat{Z}$. Roughly speaking, $J_P(V)$  is the maximal subspace of $V^{N_P^o}$ on which the $Z_{L_P}(\Q_p)^+$-action extends canonically to a locally analytique $Z_{L_P}(\Q_p)$-action. Since $L_P(\Q_p)^+Z_{L_P}(\Q_p)=L_P(\Q_p)$, $J_P(V)$ is equipped with a natural action of $L_P(\Q_p)$.
 \begin{theorem}[$\text{\cite[Prop.3.4.11,Thm.4.2.32]{Em11}}$]\label{thm: gln-npv}Keep the above notation, $J_P(V)$ is independent of the choice of $N_P^o$, and is an essentially admissible locally $\Q_p$-analytic representation of $L_P(\Q_p)$.
\end{theorem}

\begin{theorem}[$\text{\cite[Thm.5.3]{HL}}$]\label{thm: gln-sgv}
Let $P_1\subset P_2$ be two parabolic subgroups of $G$, $V$ be an essentially admissible locally analytic representation of $G(\Q_p)$, then one has a natural isomorphism of $L_{P_1}(\Q_p)$\big($=L_{P_1\cap L_{P_2}}(\Q_p)$\big)-representations
\begin{equation*}
      J_{L_{P_2} \cap P_1} (J_{P_2}(V))\xlongrightarrow{\sim} J_{P_1}(V).
\end{equation*}
\end{theorem}

\subsubsection{A digression: locally algebraic vectors}\label{sec: 2.2.1}
Denote by $G^{\cD}$ the derived subgroup of $G$, and $\ug^{\cD}$ the Lie algebra of $G^{\cD}$ over $\Q_p$. Since $G$ is reductive, $G^{\cD}$ is semisimple, and we have a local isomorphism $Z_G \times G^{\cD}\xrightarrow{\sim} G$. Let $T_{G^{\cD}}:=T\cap G^{\cD}$, which is thus a maximal split torus of $G^{\cD}$, let $\ft_{G^{\cD}}$ denote the Lie algebra of $T_{G^{\cD}}$. The inclusion $\ft_{G^{\cD}} \hookrightarrow\ft$ induces a projection $\Hom_E(\ft,E)\twoheadrightarrow \Hom_{E}(\ft_{G^{\cD}},E)$. In fact, we have an isomorphism $\ft\xrightarrow{\sim} \ft_{G^{\cD}} \times \fz_G$, and $\oplus_{\alpha\in \Delta} E \alpha \xrightarrow{\sim} \Hom_E(\ft_{G^{\cD}}, E)$. Thus a weight $\lambda$ of $\ft$ is dominant if and only if its restriction to $\ft_{G^{\cD}}$ is dominant.

For a locally $\Q_p$-analytic representation $V$ of $G(\Q_p)$ over $E$, denote by $V_0$ the $E$-vector space generated by the vectors fixed by $\ug^{\cD}$, in other words, the smooth vectors for $G^{\cD}$. It's straightforward to check $V_0$ is stable under the $G(\Q_p)$-action and is a closed subrepresentation of $V$. Let $\lambda_0$ be a dominant weight for $G^{\cD}$, $\lambda$ a weight for $G$ above $\lambda_0$ (which is thus also dominant).
Put
\begin{equation*}
  V_{\lambda_0}:=\big(V\otimes_E \sL(\lambda)'\big)_0 \otimes_E \sL(\lambda),
\end{equation*}
where $\sL(\lambda)'$ denotes the dual algebraic representation of $\sL(\lambda)$. And we would use $-\lambda$ to denote the highest weight of $\sL(\lambda)'$.
\begin{lemma}\label{lem: gln-nvd}
  Keep the above notation, $V_{\lambda_0}$ is a closed subrepresentation of $V$ of  $G(\Q_p)$ and is independent of the choice of $\lambda$, in other words, $V_{\lambda_0}$ only depends on $\lambda_0$.
\end{lemma}
\begin{proof}
One has a natural $G(\Q_p)$-invariant map
\begin{equation*}
    \big(V\otimes_E \sL(\lambda)'\big)_0 \otimes_E \sL(\lambda) \ \lra V,\ v\otimes w' \otimes w \mapsto w'(w) v.
\end{equation*}
Moreover, by \cite[Prop.4.2.4]{Em04} \big(applied to $G=G^{\cD}(\Q_p)$\big), we know this is injective. Let $\mu$ be another dominant weight which restricts to $\lambda_0$, then $\sL(\mu)$ would differ from $\sL(\lambda)$ by certain determinant twist, from which  the second part easily follows.
\end{proof}
In particular, if $V$ is essentially admissible, so is $V_{\lambda_0}$ (cf. \cite[Prop.6.4.11]{Em04}).
\begin{lemma}\label{lem: gln-atr}Let $\lambda$ be a dominant weight for $G$, $\pi$ be a smooth representation of $G(\Q_p)$ over $E$, $V$ a locally analytic representation of $G(\Q_p)$ smooth for the $G^{\cD}(\Q_p)$-action, then the following map
\begin{equation}\label{equ: gln-gip}
  \Hom_{G(\Q_p)}(\pi, V) \lra \Hom_{G(\Q_p)}\big(\pi\otimes_E \sL(\lambda), V\otimes_E \sL(\lambda)\big),\ f\mapsto f\otimes \id,
\end{equation}
is bijective.
\end{lemma}
\begin{proof}
A $G(\Q_p)$-invariant map $g: \pi \otimes_E \sL(\lambda)\ra V \otimes_E \sL(\lambda)$ induces
  \begin{equation*}\label{}
    \pi \lra \pi \otimes_E \sL(\lambda)\otimes_E \sL(\lambda)' \lra V \otimes_E \sL(\lambda) \otimes_E \sL(\lambda)';
  \end{equation*}
  since $\pi$ is smooth, this map factors through in particular 
  \begin{equation*}\pi \lra \big(V \otimes_E \sL(\lambda) \otimes_E \sL(\lambda)'\big)_0 \cong V\end{equation*}
 where the last isomorphism follows from the isomorphism above \cite[Prop.4.2.4]{Em04}. One easily sees this gives an inverse (up to scalars) of (\ref{equ: gln-gip}).
\end{proof}
\subsubsection{Subfunctors}\label{sec: 2.2.2}
Return to the situation before \S \ref{sec: 2.2.1} (in particular, $V$ is an essentially admissible locally analytic representation of $G(\Q_p)$ over $E$), and let $P_1\subseteq P_2$ be two parabolic subgroups of $G$ containing $B$. Let $\lambda_0$ be a dominant weight for $L_{P_2}^\cD$ (the derived subgroup of $L_{P_2}$), for an essentially admissible  locally analytic representation $V$ of $G(\Q_p)$, put (cf. \S \ref{sec: 2.2.1})
\begin{equation*}
  J_{P_1,(P_2,\lambda_0)}(V):=J_{P_1\cap L_{P_2}}\big(J_{P_2}(V)_{\lambda_0}\big).
\end{equation*}
By the left exactness of the functor $J_{P_1}(\cdot)$ and Lem.\ref{lem: gln-nvd}, $J_{P_1,(P_2,\lambda_0)}(V)$ is a closed subrepresentation of $J_{P_1}(V)=J_{P_1\cap L_{P_2}}(J_{P_2}(V))$. By Thm.\ref{thm: gln-npv}, \ref{thm: gln-sgv} and Lem.\ref{lem: gln-nvd}, one has
 \begin{corollary}$J_{P_1,(P_2,\lambda_0)}(V)$ is an essentially admissible locally analytic representation of $L_{P_1}(\Q_p)$.
 \end{corollary}
\begin{lemma}\label{lem: gln-att}
  Keep the above notation, let $\lambda$ be dominant weight (for $L_{P_2}$) above $\lambda_0$, $\sL_2(\lambda)$ the irreducible algebraic representation of $L_{P_2}$ with highest weight $\lambda$ we have
  \begin{equation*}
    J_{P_1,(P_2,\lambda_0)}(V)\xlongrightarrow{\sim} J_{P_1\cap L_{P_2}}\big((J_{P_2}(V) \otimes_E \sL_2(\lambda)')_0\big) \otimes_E \sL_2(\lambda)^{N_{P_1\cap L_{P_2}}}.
  \end{equation*}
\end{lemma}
\begin{proof}
  It's sufficient to prove \begin{equation}\label{equ: gln-jpv}\Big((J_{P_2}(V) \otimes_E \sL_2(\lambda)')_0 \otimes_E \sL_2(\lambda)\Big)^{N_{P_1\cap L_{P_2}}}\cong \big((J_{P_2}(V) \otimes_E \sL_2(\lambda)')_0\big)^{N_{P_1\cap L_{P_2}}} \otimes_E \sL_2(\lambda)^{N_{P_1\cap L_{P_2}}}.\end{equation} Since the action of $N_{P_1\cap L_{P_2}}$ is smooth on $(J_{P_2}(V) \otimes_E \sL_2(\lambda)')_0$, thus we have
  \begin{multline*}
    \Big((J_{P_2}(V) \otimes_E \sL_2(\lambda)')_0 \otimes_E \sL_2(\lambda)\Big)^{N_{P_1\cap L_{P_2}}}\subseteq \Big((J_{P_2}(V) \otimes_E \sL_2(\lambda)')_0 \otimes_E \sL_2(\lambda)\Big)^{\fn_{P_1\cap L_{P_2}}} \\ \subseteq (J_{P_2}(V) \otimes_E \sL_2(\lambda)')_0\otimes_E \sL_2(\lambda)^{\fn_{P_1\cap L_{P_2}}}=(J_{P_2}(V) \otimes_E \sL_2(\lambda)')_0\otimes_E \sL_2(\lambda)^{N_{P_1\cap L_{P_2}}}
  \end{multline*}
  where $\fn_{P_1\cap L_{P_2}}$ denotes the Lie algebra of $N_{P_1\cap L_{P_2}}$, from which we easily deduce that the left side is contained in the right side in (\ref{equ: gln-jpv}). The other direction is trivial, and the lemma follows.
\end{proof}
Keep the above notation, and denote by $\sL_1(\lambda)$ the irreducible algebraic representation of $L_{P_1}(\Q_p)$ with highest weight $\lambda$, thus $\sL_1(\lambda)\cong \sL_2(\lambda)^{N_{P_1\cap L_{P_2}}}$. Note that the action of $L_{P_1}(\Q_p) \cap L_{P_2}^{\cD}(\Q_p)$ on $J_{P_1\cap L_{P_2}}\big((J_{P_2}(V) \otimes_E \sL_2(\lambda)')_0\big) $ is smooth, thus $J_{P_1,(P_2,\lambda_0)}(V)$ is a locally algebraic representation of $L_{P_1}^{\cD}(\Q_p)$ \big(of type $\sL_1(\lambda)$\big). Let $H$ be an open compact subgroup of $\big(Z_{L_{P_1}}\cap L_{P_2}^{\cD}\big)(\Q_p)$, $\chi$ be a smooth character of $H$ over $E$, and put
\begin{equation}\label{equ: gln-ivP}
  J_{P_1, (P_2,\lambda_0)}^{H,\chi}(V):=J_{P_1\cap L_{P_2}}\big((J_{P_2}(V)\otimes_E \sL_2(\lambda)')_0\big)^{H=\chi}\otimes_E \sL_1(\lambda)\hooklongrightarrow J_{P_1,(P_2,\lambda_0)}(V),
\end{equation}
which is also an essentially admissible locally analytic representation of $L_{P_1}(\Q_p)$.
\subsection{Adjunction formulas}In this section, we deduce from \cite[Thm.4.3]{Br13II} an adjunction formula for the functor $J_{P_1,(P_2,\lambda_0)}(\cdot)$, which would play a crucial role in our study below of certain closed rigid subspaces of the eigenvarieties.

Denote by $\delta_i$ the modulus character of $P_i(\Q_p)$ for $i=1,2$, $\delta_{12}$ the modulus character of $P_1(\Q_p) \cap L_{P_2}(\Q_p)$ \big(where $P_1\cap L_{P_2}$ is viewed as a parabolic subgroup of $L_{P_2}$\big). Note the character $\delta_i$ factors through $L_{P_i}$, and $\delta_{12}$ factors through $L_{P_1}$. We have $\delta_{12}|_{L_{P_1}}\delta_2|_{L_{P_1}}=\delta_1|_{L_{P_1}}$.

Let $W$ be an irreducible algebraic representation of $L_{P_1}$ of highest weight $\mu$, $\pi$ be a finite length smooth admissible representation of $L_{P_1}(\Q_p)$, suppose there exists a non-zero $L_{P_1}(\Q_p)$-invariant map: $\pi \otimes_E W \ra J_{P_1,(P_2,\lambda_0)}(V)$. Then one has
\begin{lemma}
  The weight $\mu$ is $P_2$-dominant and restricts to $\lambda_0$.
\end{lemma}
\begin{proof}
  By considering the action of the Lie algebra of $L_{P_1}(\Q_p)\cap L^{\cD}_{P_2}(\Q_p)$  and Lem.\ref{lem: gln-att}, there exists a $P_2$-dominant weight $\lambda$ above $\lambda_0$ such that $W|_{L_{P_1}(\Q_p)\cap L^{\cD}_{P_2}(\Q_p)} \cong \sL_1(\lambda)|_{L_{P_1}(\Q_p)\cap L^{\cD}_{P_2}(\Q_p)}$. Since $L_{P_1}^{\cD}\subseteq  L_{P_1} \cap L_{P_2}^{\cD}$, we see $W$ differs from $\sL_1(\lambda)$ by certain determinantal twist $\dett_1$. Since $\dett_1$ is trivial on $L_{P_1}(\Q_p) \cap L^{\cD}_{P_2}(\Q_p)$ \big(thus trivial on $(T\cap L^{\cD}_{P_2})(\Q_p)$\big), the weight $\dett_1$ is $P_2$-dominant. The lemma follows.
\end{proof}For a parabolic subgroup $P$ of $G$, we use $\overline{P}$ to denote the parabolic subgroup of $G$ opposite to $P$.
\begin{proposition}\label{prop: gln-trs}
  Let $\lambda$ be a dominant weight for $L_{P_2}$ which restricts to $\lambda_0$, $\pi$ be a finite length smooth admissible representation of $L_{P_1}(\Q_p)$, $V$ be an essentially admissible locally analytic representation of $L_{P_2}(\Q_p)$, then one has a natural bijection
\begin{multline}\label{equ: gln-g2p}
    \Hom_{L_{P_2}(\Q_p)}\Big(\big(\Ind_{\overline{P}_1(\Q_p)\cap L_{P_2}(\Q_p)}^{L_{P_2}(\Q_p)} \pi\otimes_E \delta_{12}^{-1}\big)^{\infty}\otimes_E \sL_2(\lambda) , V_{\lambda_0}\Big) \\ \xlongrightarrow{\sim} \Hom_{L_{P_1}(\Q_p)}\Big( \pi\otimes_E \sL_1(\lambda) , J_{P_1\cap L_{P_2}}(V_{\lambda_0})\Big).
\end{multline}
\end{proposition}
\begin{proof}This proposition follows  from the adjunction property for classical Jacquet functors. Recall $V_{\lambda_0}\cong (V\otimes_E \sL_2(\lambda)')_0 \otimes_E \sL_2(\lambda)$.  By Lem.\ref{lem: gln-atr}, \ref{lem: gln-att}, one has
 \begin{equation}\label{equ: gln-ip2}
 \Hom_{L_{P_1}(\Q_p)}\big(\pi, J_{P_1\cap L_{P_2}}\big((V \otimes_E \sL_2(\lambda)')_0\big)\big) \xlongrightarrow{\sim}
    \Hom_{L_{P_1}(\Q_p)}\big(\pi\otimes_E \sL_1(\lambda), J_{P_1\cap L_{P_2}}(V_{\lambda_0})\big).
 \end{equation}
Since $\pi$ is smooth and of finite length, $Z_{L_{P_1}}(\Q_p)$ \big(thus $Z_{L_{P_2}}(\Q_p)$\big) acts on $\pi$ via certain finite dimensional representation, in other words, there exists an ideal $\cI$ of $E[Z_{L_{P_2}}(\Q_p)]$, $\dim_E (E[Z_{L_{P_2}}(\Q_p)]/\cI) < \infty$, such that $\pi=\pi^{I}$ \big(where $\pi$ is viewed as an $E[Z_{L_{P_2}}(\Q_p)]$-module\big). So
  \begin{equation*}\label{equ: gln-nsew}
    \Hom_{L_{P_1}(\Q_p)}\Big(\pi, J_{P_1\cap L_{P_2}}\big((V\otimes_E \sL_2(\lambda)')_0^{\cI}\big)\Big) \xlongrightarrow{\sim} \Hom_{L_{P_1}(\Q_p)}\Big(\pi,  J_{P_1\cap L_{P_2}}\big((V\otimes_E \sL_2(\lambda)')_0\big)\Big).
  \end{equation*}
  On the other hand, by \cite[Prop.6.4.13]{Em04} and \cite[Thm.6.6]{ST03}, $(V\otimes_E \sL_2(\lambda)')_0^{\cI}$ is an admissible smooth representation of $L_{P_2}(\Q_p)$. In this case, the Jacquet-Emerton functor coincides with the classical Jacquet functor (cf. \cite[\S 4.3]{Em11}), and one has a bijection
  \begin{multline}\label{equ: gln-dll}
    \Hom_{L_{P_2}(\Q_p)}\Big(\big(\Ind_{\overline{P_1}(\Q_p)\cap L_{P_2}(\Q_p)}^{L_{P_2}(\Q_p)} \pi\otimes_E \delta_{12}^{-1}\big)^{\infty}, (V\otimes_E \sL_2(\lambda)')_0^{\cI}\Big) \\
    \xlongrightarrow{\sim} \Hom_{L_{P_1}(\Q_p)}\Big(\pi, J_{P_1\cap L_{P_2}}\big( (V\otimes_E \sL_2(\lambda)')_0^{\cI}\big)\Big),
  \end{multline}
  Note one can remove ``$\cI$" on either side since the corresponding set would not change.
  Again by Lem.\ref{lem: gln-atr}, one has a bijection
  \begin{multline}\label{equ: gln-lgii}
      \Hom_{L_{P_2}(\Q_p)}\Big(\big(\Ind_{\overline{P_1}(\Q_p)\cap L_{P_2}(\Q_p)}^{L_{P_2}(\Q_p)} \pi\otimes_E \delta_{12}^{-1}\big)^{\infty}, (V\otimes_E \sL_2(\lambda)')_0\Big) \\ \xlongrightarrow{\sim} \Hom_{L_{P_2}(\Q_p)}\Big(\big(\Ind_{\overline{P_1}(\Q_p)\cap L_{P_2}(\Q_p)}^{L_{P_2}(\Q_p)} \pi\otimes_E \delta_{12}^{-1}\big)^{\infty}\otimes_E \sL_2(\lambda), (V\otimes_E \sL_2(\lambda)')_0\otimes_E \sL_2(\lambda)\Big).
  \end{multline}
  Putting (\ref{equ: gln-ip2}) (\ref{equ: gln-dll}) (\ref{equ: gln-lgii}) together, the proposition follows.
\end{proof}
\begin{remark}\label{rem: gln-tfe}Note the left set of (\ref{equ: gln-g2p}) won't change if  $V_{\lambda_0}$ is replaced by $V$ since any non-zero $L_{P_2}(\Q_p)$-invariant map
 \begin{equation*}
 \big(\Ind_{\overline{P}_1(\Q_p)\cap L_{P_2}(\Q_p)}^{L_{P_2}(\Q_p)} \pi\otimes_E \delta_{12}^{-1}\big)^{\infty}\otimes_E \sL_2(\lambda)\lra V
 \end{equation*}
factors automatically through $V_{\lambda_0}$. However, the set on the right side is rather subtile, indeed, the natural injection
\begin{equation*}
  \Hom_{L_{P_1}(\Q_p)}\Big( \pi\otimes_E \sL_1(\lambda) , J_{P_1\cap L_{P_2}}(V_{\lambda_0})\Big) \hooklongrightarrow \Hom_{L_{P_1}(\Q_p)}\Big( \pi\otimes_E \sL_1(\lambda) , J_{P_1\cap L_{P_2}}(V)\Big)
\end{equation*}
is \emph{not} bijective in general.
\end{remark}
\begin{theorem}\label{thm: gln-ycn}
  Let $V$ be a very strongly admissible locally $\Q_p$-analytic representation of $G(\Q_p)$ (resp. \cite[Def.0.12]{Em2}), $\lambda$ a dominant weight for $L_{P_2}$ above $\lambda_0$, $\pi$ a finite length smooth admissible representation of $L_{P_1}(\Q_p)$, thus there exists a natural bijection
  \begin{multline}\label{equ: gln-e2G}
    \Hom_{G(\Q_p)}\Big(\cF_{\overline{P}_2}^{G}\Big( \big(\text{U}(\ug) \otimes_{\text{U}(\overline{\fp}_2)}\sL_2(\lambda)'\big)^{\vee}, \big(\Ind_{\overline{P}_1(\Q_p)\cap L_{P_2}(\Q_p)}^{L_{P_2}(\Q_p)} \pi \otimes_E \delta_1^{-1}\big)^{\infty}\Big), V\Big)
    \\
    \xlongrightarrow{\sim} \Hom_{L_{P_1}(\Q_p)}\big( \pi\otimes_E \sL_1(\lambda), J_{P_1,(P_2,\lambda_0)}(V)\big),
  \end{multline}
  where $(\cdot)^{\vee}$ denotes the dual in the category $\co^{\overline{\ub}}$ (cf. \cite[Chap.3]{Hum08}).
\end{theorem}
\begin{proof}By Prop.\ref{prop: gln-trs} (applied to $V=J_{P_2}(V)$), one has a bijection 
  \begin{multline*}
    \Hom_{L_{P_2}(\Q_p)}\Big(\big(\Ind_{\overline{P}_1(\Q_p)\cap L_{P_2}(\Q_p)}^{L_{P_2}(\Q_p)} \pi\otimes_E \delta_{12}^{-1}\big)^{\infty}\otimes_E \sL_2(\lambda), J_{P_2}(V)_{\lambda_0}\Big) \\ \xlongrightarrow{\sim} \Hom_{L_{P_1}(\Q_p)}\Big( \pi\otimes_E \sL_1(\lambda), J_{P_1\cap L_{P_2}}(J_{P_2}(V)_{\lambda_0})\Big),
\end{multline*}
and the left set would not change if $J_{P_2}(V)_{\lambda_0}$ is replaced by $J_{P_2}(V)$ (see Rem.\ref{rem: gln-tfe}).
Since $V$ is very strongly admissible, by \cite[Thm.4.3]{Br13II}, one has
\begin{multline*}
    \Hom_{G(L)}\Big(\cF_{\overline{P}_2}^{G}\Big( \big(\text{U}(\ug) \otimes_{\text{U}(\overline{\fp}_2)} \sL_2(\lambda)'\big)^{\vee}, \big(\Ind_{\overline{P}_1(\Q_p)\cap L_{P_2}(\Q_p)}^{L_{P_2}(\Q_p)} \pi \otimes_E \delta_{12}^{-1}\big)^{\infty}\otimes_E \delta_2^{-1}\Big), V\Big)
    \\
    \xlongrightarrow{\sim}    \Hom_{L_{P_2}(\Q_p)}\Big(\big(\Ind_{\overline{P}_1(\Q_p)\cap L_{P_2}(\Q_p)}^{L_{P_2}(\Q_p)} \pi\otimes_E \delta_{12}^{-1}\big)^{\infty}\otimes_E \sL_2(\lambda), J_{P_2}(V)\Big).
  \end{multline*}
Since $\delta_1|_{L_{P_1}(\Q_p)}=\delta_{12}|_{L_{P_1}(\Q_p)}\delta_2|_{L_{P_1}(\Q_p)}$, the theorem follows.
\end{proof}

\begin{remark}Keep the notations in Thm.\ref{thm: gln-ycn}. By \cite[Thm.4.3]{Br13II}, one has a bijection
\begin{equation}\label{equ: gln-1Gp}
    \Hom_{G(\Q_p)}\Big(\cF_{\overline{P}_1}^{G}\Big( \big(\text{U}(\ug) \otimes_{\text{U}(\overline{\fp}_1)} \sL_1(\lambda)'\big)^{\vee}, \pi\otimes_E \delta_1^{-1}\Big), V\Big)
    \xlongrightarrow{\sim} \Hom_{L_{P_1}(\Q_p)}\Big( \pi\otimes_E \sL_1(\lambda), J_{P_1}(V)\Big).
\end{equation}
Denote by $M_i(\lambda):=\text{U}(\ug) \otimes_{\text{U}(\overline{\fp}_i)} \sL_i(\lambda)'$. Note $M_2(\lambda)$ is a quotient of $M_1(\lambda)$ in the category $\co^{\overline{\fp}_1}$ (cf. \cite[\S 9]{Hum08}), from which we deduce (by Thm.\ref{thm: gln-pst} (2) and (3)) the representation
\begin{equation*}
  \cF_{\overline{P}_2}^{G}\Big( \big(\text{U}(\ug) \otimes_{\text{U}(\overline{\fp}_2)} \sL_2(\lambda)'\big)^{\vee}, \big(\Ind_{\overline{P}_1(\Q_p)\cap L_{P_2}(\Q_p)}^{L_{P_2}(\Q_p)} \pi \otimes_E \delta_{12}^{-1}\big)^{\infty}\otimes_E \delta_2^{-1}\Big)
\end{equation*}is a quotient of $\cF_{\overline{P}_1}^{G}\Big( \big(\text{U}(\ug) \otimes_{\text{U}(\overline{\fp}_1)} \sL_1(\lambda)'\big)^{\vee}, \pi\otimes_E \delta_1^{-1}\Big)$. In particular, when $P_1 \neq P_2$, the right set of (\ref{equ: gln-1Gp}) might be strictly bigger than the right set of (\ref{equ: gln-e2G}). This subtlety is somehow the key point of this note.
\end{remark}

\subsection{Structure of $J_{B,(P,\lambda_0) }(V)$}\label{sec: gln-2.4} Let $H$ be an open compact uniform prop-$p$-subgroup of $G(\Q_p)$ such that $H$ is a normal subgroup of the maximal open compact subgroup of $G(\Q_p)$. Let $V$ be a locally $\Q_p$-analytic representation of $G(\Q_p)$  over $E$ such that $V|_H\cong \cC^{\la}(H,E)^{\oplus r}$ for $r\in \Z_{\geq 1}$. This section is devoted to the structure $J_{B,(P,\lambda_0)}(V)$. In fact, the  results in this section were already obtained in \cite{HL}, while we reformulate them in a way that suits our context.

For any parabolic subgroup $P'$ of $G$, let $N_{P'}^o:=H\cap N_{P'}$, $L_{P'}^o:=H\cap L_{P'}^o$, put $L_{P'}(\Q_p)^+$ as (\ref{equ: gln-1np}) with respect to $N_{P'}^o$. Note $T\cong L_B$, and put $N:=N_B$, $\overline{N}:=N_{\overline{B}}$, $N^o:=N_B^o$, $\overline{N}^o:=N_{\overline{B}}^o$, $T^o:=L_B^o$ and $T(\Q_p)^+:=L_B(\Q_p)^+$.

Suppose the natural multiplication induces a homemorphism
\begin{equation}\label{equ: gln-nwo}
  N_{*}^o \times L_*^o \times N_{\overline{*}}^o \xlongrightarrow{\sim} H
\end{equation}
for $*\in \{P, B\}$. Denote by $L_P^{\cD,o}:=L_P^{\cD}(\Q_p)\cap L_P^o$, $Z_{L_P}^{o}:=Z_{L_P}(\Q_p) \cap L_P^{o}$, and suppose the following natural map is an isomorphism of $p$-adic analytic groups
\begin{equation}\label{equ: gln-omo}
  L_P^{\cD,o}\times Z_{L_P}^{o} \xlongrightarrow{\sim} L_P^{o}.
\end{equation}
Let $\lambda$ be a dominant weight for $G$ above $\lambda_0$, we use $\sL_P(\lambda)$ \big(resp. $\sL_G(\lambda)$\big) to denote the irreducible algebraic representation of $P$ (resp. $G$) over $E$ with highest weight $\lambda$.  Recall \begin{equation*}J_{B,(P,\lambda_0)}(V)\cong \Big(\big(V^{N_P^o}_{\cfs} \otimes_E \sL_P(\lambda)'\big)_0 \otimes_E \sL_P(\lambda)\Big)^{N_{B\cap L_P}^{o}}_{\cfs},\end{equation*}where $N_{B\cap L_P}^o:=N_{B\cap L_P}\cap H$, the first ``$\cfs$" is with respect to $Z_{L_P}(\Q_p)^+$, and second one to $T(\Q_p)^+=Z_{L_B}(\Q_p)^+$. Note $Z_{L_P}(\Q_p)^+\subseteq T(\Q_p)^+$.
\begin{lemma}
$J_{B,(P,\lambda_0)}(V)\cong \Big(\big(V^{N_P^{o}}\otimes_E \sL_P(\lambda)'\big)_0 \otimes_E \sL_P(\lambda)\Big)^{N_{B\cap L_P}^{o}}_{\cfs}$, where $\big(V^{N_P^{o}}\otimes_E \sL_P(\lambda)'\big)_0:=\varinjlim_{U} \big(V^{N_P^{o}}\otimes_E \sL_P(\lambda)'\big)^U$ with $U$ running through the open compact subgroups of $L_P^{\cD,o}$.\end{lemma}
\begin{proof}
  Since $V^{N_P^o}_{\cfs}$ is a closed subspace of $V^{N_P^o}$, we have
  \begin{equation*}
   J_{B,(P,\lambda_0)}(V)\hooklongrightarrow \Big(\big(V^{N_P^{o}}\otimes_E \sL_P(\lambda)'\big)_0 \otimes_E \sL_P(\lambda)\Big)^{N_{B\cap L_P}^{o}}_{\cfs}.
  \end{equation*}We prove this map is bijective. Since $\sL_P(\lambda)^{N_{B\cap L_P}^o}\cong \chi_{\lambda}$ (where $\chi_{\lambda}$ denotes the algebraic character of $T$ with weight $\lambda$), we reduce to prove
  \begin{equation*}
     \Big(\big(V^{N_P^{o}}_{\cfs}\otimes_E \sL_P(\lambda)'\big)_0\Big)^{N_{B\cap L_P}^o}_{\cfs} \hooklongrightarrow  \Big(\big(V^{N_P^{o}}\otimes_E \sL_P(\lambda)'\big)_0\Big)^{N_{B\cap L_P}^o}_{\cfs}
  \end{equation*}
  is bijective.
  Since $\sL_P(\lambda)'\cong \big(\sL_G(\lambda)'\big)^{N_P^o}=\big(\sL_G(\lambda)'\big)^{N_P^o}_{\cfs}$, thus $\big(V^{N_P^{o}}_{*}\otimes_E \sL_P(\lambda)'\big)_0\cong \big(\big(V\otimes_E \sL_G(\lambda)'\big)^{N_P^o}_{*}\big)_0$ with $*\in \{\cfs, \emptyset\}$. We reduce to prove the natural injection
  \begin{equation}\label{equ: gln-bpo}
    \big(\big(W^{N_P^o}_{\cfs}\big)_0\big)^{N_{B\cap L_P}^o}_{\cfs} \hooklongrightarrow \big(\big(W^{N_P^o}\big)_0\big)^{N_{B\cap L_P}^o}_{\cfs}
  \end{equation}
  is bijective for any essentially admissible locally analytic representation $W$ of $G(\Q_p)$. Firstly, note $\big(W^{N_P^o}_{\cfs}\big)_0\cong \big(\big(W^{N_P^o}\big)_0\big)_{\cfs}$ since the operation $(\cdot)_0$ depends only on the action of $L_P^{\cD,o}$ and thus commutes with $(\cdot)_{\cfs}$ (cf. \cite[Prop.3.2.11]{Em11}).  Consider the inclusion map, $\big(\big(W^{N_P^o}\big)_0\big)^{N_{B\cap L_P}^o}_{\cfs}\hookrightarrow \big(W^{N_P^o}\big)_0$, by the universal property \cite[Prop.3.2.4 (ii)]{Em11}, we see this map factors through $\big(\big(W^{N_P^o}\big)_0\big)^{N_{B\cap L_P}^o}_{\cfs}\hookrightarrow \big(\big(W^{N_P^o}\big)_0\big)_{\cfs}\cong  \big(W^{N_P^o}_{\cfs}\big)_0$, whose image is thus contained in $ \big(\big(W^{N_P^o}_{\cfs}\big)_0\big)^{N_{B\cap L_P}^o}_{\cfs}$. This gives an inverse of (\ref{equ: gln-bpo}).
\end{proof}
By this lemma (and the proof), \begin{equation*}J_{B,(P,\lambda_0)}(V)\cong \Big(\big(V^{N_P^{o}}\otimes_E \sL_P(\lambda)'\big)_0 \otimes_E \sL_P(\lambda)\Big)^{N_{B\cap L_P}^{o}}_{\cfs}\cong \Big(\big(V^{N_P^{o}}\otimes_E \sL_P(\lambda)'\big)_0\Big)^{N_{B\cap L_P}^{o}}_{\cfs}\otimes_E \chi_{\lambda}, \end{equation*}in the following, we would consider $\Big(\big(V^{N_P^{o}}\otimes_E \sL_P(\lambda)'\big)_0\Big)^{N_{B\cap L_P}^{o}}_{\cfs}$, which only  differs from $J_{B,(P,\lambda_0)}(V)$ by the twist $\chi_{\lambda}$. The following lemma is well known.
\begin{lemma}\label{lem: gln-nec}Let  $W$ be an irreducible algebraic representation of $G$ over $E$, then one has an isomorphism of $H$-representations
  $\cC^{\la}(H,E) \otimes_E W \cong \cC^{\la}(H,E)^{\oplus \dim_E W}$, where $H$ acts on the left object via diagonal action.
\end{lemma}
\begin{proof}We include a proof for the convenience  of the reader.
  We first prove $\cC(H,E) \otimes_E W \cong \cC(H,E)^{\oplus \dim_E W}$. Since $H$ is pro-$p$, $(\co_E/\varpi_E^n)[[H]]$ is a complete local algebra, any finitely generated projective $(\co_E/\varpi_E^n)[[H]]$-module is isomorphic to $(\co_E/\varpi_E^n)[[H]]^{\oplus r}$ for $r\in \Z_{\geq 1}$. Thus by dualizing, any smooth admissible $H$-representation over $\co_E/\varpi_E^n$ which is moreover injective,  is isomorphic to $\cC(H,\co_E/\varpi_E)^{\oplus r}$ for $r\in \Z_{\geq 1}$. Let $W_0$ be a $H$-invariant $\co_E$-lattice of $W$, let $V_n:=\cC(H,\co_E/\varpi_E^n) \otimes_{\co_E/\varpi_E^n} W_0/\varpi_E^n$, we claim $V_n$ is an injective object in the category $\cG_n$ of smooth admissible $H$-representations over $\co_E/\varpi_E^n$.


   Indeed, for any $M\in \cG_n$, we have\begin{equation*}
    \Hom_{\co_E/\varpi_E^n}\big(M, \cC(H,\co_E/\varpi_E^n)\big)\cong\cC(H,\co_E/\varpi_E^n) \otimes_{\co_E/\varpi_E^n} M^{\vee}\cong \cC(H,M^{\vee}),
  \end{equation*}which is moreover $H$-invariant, where the $H$-action on the left object is given by $h(f)(m):=h^{-1}(f(h(m)))$, on the middle one is via diagonal action, and on $\cC(H,M^{\vee})$ is given $h(f)(h')=h^{-1}f(h'h)$, thus one has $\Hom_H\big(M,\cC(H,\co_E/\varpi_E^n)\big)\cong M^{\vee}$, $f\mapsto f(1)$. Consequently, \begin{equation*}
    \Hom_H\big(M, V_n\big)=\Hom_H\big(M\otimes_{\co_E/\varpi_E^n} (W_0/\varpi_E^n)^{\vee}, \cC(H,\co_E/\varpi_E^n)\big),
  \end{equation*}
  since the functor $-\otimes_{\co_E/\varpi_E^n} (W_0/\varpi_E^n)^{\vee}$ is exact, we deduce $V_n$ is injective from the injectivity of $\cC(H,\co_E/\varpi_E^n)$.

  Thus there exists $r\in \Z_{\geq 1}$, such that $V_n\cong \cC(H,\co_E/\varpi_E^n)^{\oplus r}$. On the other hand, let $H'$ be an open compact subgroup of $H$ which acts trivially on $W_0/\varpi_E^n$, and we have thus $\big(\cC(H,\co_E/\varpi_E^n)^{H'}\big)^{\oplus r}\cong V_n^{H'}\cong \big(\cC(H,\co_E/\varpi_E^n)^{H'}\big) \otimes_{\co_E/\varpi_E^n} W_0/\varpi_E^n$ (note these are all finite sets), from which we see $r=\dim_E W$. By taking projective limit on $n$ and tensoring with $E$, we get $\cC(H,E)\otimes_E W\cong \cC(H,E)^{\oplus \dim_E W}$.

  By \cite[Prop.3.6.15]{Em11}, $(\cC(H,E)\otimes_E W)_{\an}\cong \cC^{\la}(H,E) \otimes_E W$, the lemma follows.
\end{proof}
For a $p$-adic analytic group $H'$, denote by $\cD(H',E):=\cC^{\la}(H',E)^{\vee}_b$ the $E$-algebra of distributions of $H'$, which is a Fr\'echet-Stein algebra when $H'$ is compact.
Let $\cC^{\infty}(H',E)\hookrightarrow \cC^{\la}(H',E)$ be the (closed) subspace of smooth functions, i.e. functions killed by the Lie algebra action. Put $\cD^{\infty}(H',E):=\cC^{\infty}(H',E)^{\vee}_b$ which is a closed quotient of $\cD(H',E)$ and thus is also a Fr\'echet-Stein algebra when $H'$ is compact. In fact, one has an isomorphism of topological $E$-algebras $\cD^{\infty}(H',E)\cong \varprojlim_{U\subseteq H'} E[H'/U]$ with $U$ running over open compact normal subgroups of $H'$ (cf. \cite[\S 2]{ST01f}).

Let $T_P^o:=T(\Q_p) \cap L_P^{\cD,o}$. By (\ref{equ: gln-omo}), $Z_{L_P}^o \times L_P^{\cD,o}\xrightarrow{\sim}L_P^o$, thus
\begin{equation}\label{equ: gln-pzp}T^o=T(\Q_p)\cap L_P^o\cong Z_{L_P}^o\times (L_P^{\cD,o}\cap T(\Q_p))=Z_{L_P}^o\times T_P^o.\end{equation}
This isomorphism induces thus an isomorphism of Fr\'echet-Stein algebras
\begin{equation*}
  \cD(T_P^o,E)\widehat{\otimes}_E \cD(Z_{L_P}^o,E) \xlongrightarrow{\sim} \cD(T^o,E).
\end{equation*}
Put
\begin{equation*}
  \cD'(T^o,E):=\cD^{\infty}(T^o_P, E) \widehat{\otimes}_E \cD(Z_{L_P}^o,E)
\end{equation*}
which is thus a quotient of $\cD(T^o,E)$ and is also a Fr\'echet-Stein algebra. 


Recall for a compact uniform prop-$p$-group $H'$, and for $\frac{1}{p} <r<1$, $\cD(H',E)$ is equipped with a multiplicative norm $\|\cdot\|_r$ (\cite[\S 4]{ST03}). As in \emph{loc. cit.}, denote by $\cD_r(H',E)$ the completion of $\cD(H',E)$ via $\|\cdot\|_r$, which is thus a Banach $E$-algebra. We can also define a bigger Banach $E$-algebra $\cD_{< r}(H',E)$ (\cite[p.161]{ST03}). One has $\cD(H',E)\cong \varprojlim_{r} \cD_r(H',E)\cong \varprojlim_{r} \cD_{<r}(H',E)$. For $n\in \Z_{\geq 1}$, put $r_n:=\frac{1}{p^{n}(p-1)}$. For a locally analytic representation $W$ of $H'$, let $W^{(n)}$ denote the subrepresentation generated by $r_n$-analytic vectors as in \cite[\S 0.3]{CD}. Put $\cC^{(n)}(H',E):=\cC^{\la}(H',E)^{(n)}$. By definition, one has $\cC^{(n)}(H',E)^{\vee}_b \cong \cD_{<r_n}(H',E)$.

Since $H$ is uniform, by the isomorphisms (\ref{equ: gln-nwo}) (\ref{equ: gln-omo}), so are the groups $T^o$, $L_P^o$, $N_P^o$, $N_{\overline{P}}^o$, $N^o$, $\overline{N}^o$, $T_P^o$, $L_P^{\cD,o}$, $Z_{L_P}^o$.
For $\frac{1}{p}< r < 1$, put $\cD_*^{\infty}(T^o_P,E)$ to be the image of $\cD_*(T^o_P,E)$ via the projection $\cD(T^o_P,E)\twoheadrightarrow \cD^{\infty}(T^o_P,E)$ for $*\in \{r, <r\}$.
Put $A_n:=\cD'_{r_n}(T^o,E):=D^{\infty}_{r_n}(T^o_P,E) \widehat{\otimes}_E \cD_{r_n}(Z_{L_P}^o,E)$. Let $z\in T(\Q_p)^+$ such that $T(\Q_p)$ is generated by $z^{-1}$ and $T(\Q_p)^+$ by multiplication.
\begin{proposition}
For all $n\in \Z_{\geq 1}$, there exists an  orthonormalisable $A_n$-module $M_n$ such that
  \begin{enumerate}
  \item $M_n$ is equipped with a compact $A_n$-linear operator $z_n$;
  \item there exist continuous $A_n$-linear maps $\alpha_n: M_n \ra M_{n+1} \otimes_{A_{n+1}} A_n$ and $\beta_n: M_{n+1} \widehat{\otimes}_{A_{n+1}}A_n \ra M_n$ such that $\beta_n \circ \alpha_n=z_n$, $\alpha_n \circ \beta_n=z_{n+1}\otimes 1_{A_n}$;
  \item one has an isomorphism of $\cD'(T^o,E)$-modules:
  \begin{equation*}\varprojlim_n M_n \xlongrightarrow{\sim} \Big(\big(\big(V^{N_P^o} \otimes_E \sL_P(\lambda)'\big)_0\big)^{N_{B\cap L_P}^o}\Big)^{\vee}
  \end{equation*}
  \big(where the maps in the projective system are given by the composition of $\beta_n$ and the natural map $M_{n+1} \ra M_{n+1} \widehat{\otimes}_{A_{n+1}} A_n$\big), which commutes with the action $\{z_n\}$ on the left and $\pi_z$ on the right (cf. (\ref{equ: gln-gnp})).
      \end{enumerate}
In summary, one has the following commutative diagram
\begin{equation*}
  \begindc{\commdiag}[40]
  \obj(0,1)[a]{$ \Big(\big(\big(V^{N_P^o} \otimes_E \sL_P(\lambda)'\big)_0\big)^{N_{B\cap L_P}^o}\Big)^{\vee}$}
  \obj(3,1)[b]{$\cdots$}
  \obj(4,1)[c]{$M_{n+1}$}
  \obj(6,1)[d]{$M_{n+1}\otimes_{A_{n+1}} A_n$}
  \obj(8,1)[e]{$M_n$}
    \obj(0,0)[a']{$ \Big(\big(\big(V^{N_P^o} \otimes_E \sL_P(\lambda)'\big)_0\big)^{N_{B\cap L_P}^o}\Big)^{\vee}$}
  \obj(3,0)[b']{$\cdots$}
  \obj(4,0)[c']{$M_{n+1}$}
  \obj(6,0)[d']{$M_{n+1}\otimes_{A_{n+1}} A_n$}
  \obj(8,0)[e']{$M_n$}
  \mor{a}{b}{}
  \mor{b}{c}{}
  \mor{c}{d}{}
  \mor{d}{e}{$\beta_n$}
    \mor{a'}{b'}{}
  \mor{b'}{c'}{}
  \mor{c'}{d'}{}
  \mor{d'}{e'}{$\beta_n$}
  \mor{a}{a'}{$\pi_z$}
  \mor{c}{c'}{$z_{n+1}$}
  \mor{d}{d'}{$z_{n+1}\otimes \id$}[\atright,\solidarrow]
  \mor{e}{e'}{$z_n$}
  \mor{e}{d'}{$\alpha_n$}
  \enddc.
\end{equation*}
\end{proposition}
\begin{proof}
We use the argument of \cite[Prop.5.3]{BHS1}, which is rather a variation of the arguments in \cite[\S 4.2]{Em11}. One has $V^{N_P^o} \otimes_E \sL_P(\lambda)'\cong (V\otimes_E \sL_G(\lambda)')^{N_P^o}$. Denote by $\Pi:=V\otimes_E \sL_G(\lambda)'$, by Lem.\ref{lem: gln-nec}, $\Pi|_H\cong \cC^{\la}(H,E)^{\oplus r}$ for some $r\in \Z_{\geq 1}$. Denote by $\Pi_H^{(n)}$ the $r_n$-analytic vectors for the $H$-action (cf. \cite[\S 0.3]{CD}), thus $\Pi_H^{(n)}\cong \cC^{(n)}(H,E)^{\oplus r}$. By the isomorphism (\ref{equ: gln-nwo}), one has
\begin{equation*}
  \Pi_H^{(n)}\xlongrightarrow{\sim} \Big(\cC^{(n)}(N_{\overline{P}}^o,E) \widehat{\otimes}_E \cC^{(n)}(L_P^{o},E) \widehat{\otimes}_E \cC^{(n)}(N_P^o,E)\Big)^{\oplus r},
\end{equation*}thus
\begin{equation*}
  \big(\Pi_H^{(n)}\big)^{N_P^o}\xlongrightarrow{\sim} \Big(\cC^{(n)}(N_{\overline{P}}^o,E) \widehat{\otimes}_E \cC^{(n)}(L_P^o,E)\Big)^{\oplus r}\\
  \xlongrightarrow{\sim} \Big(\cC^{(n)}(N_{\overline{P}}^o,E)\widehat{\otimes}_E \cC^{(n)}(L_P^{\cD,o},E)\widehat{\otimes}_E \cC^{(n)}(Z_{L_P}^o,E)\Big)^{\oplus r}.
\end{equation*}
Put $\cC^{\infty,(n)}(H',E):=\cC^{(n)}(H',E) \cap \cC^{\infty}(H',E)$, for a compact prop-$p$ uniform $p$-adic analytic group $H'$. One gets
\begin{multline*}
   \big(\big(\Pi_H^{(n)}\big)^{N_P^o}\big)_0:= \big(\Pi_H^{(n)}\big)^{N_P^o}\cap (\Pi^{N_P^o})_0\xlongrightarrow{\sim} \Big(\cC^{(n)}(N_{\overline{P}}^o,E)\widehat{\otimes}_E \cC^{(n)}(Z_{L_P}^o,E)\widehat{\otimes}_E \cC^{\infty,(n)}(L_P^{\cD,o},E)\Big)^{\oplus r} \\
   \xlongrightarrow{\sim}    \Big(\cC^{(n)}(N_{\overline{P}}^o,E)\widehat{\otimes}_E \cC^{(n)}(Z_{L_P}^o,E)\widehat{\otimes}_E \cC^{\infty,(n)}(N_{\overline{B}\cap L_P}^o,E) \widehat{\otimes}_E \cC^{\infty,(n)}(T_P^o,E) \widehat{\otimes}_E \cC^{\infty,(n)}(N_{B\cap L_P}^o,E)\Big)^{\oplus r},
\end{multline*}
and thus
\begin{multline}
  \Big(\big(\big(\Pi_H^{(n)}\big)^{N_P^o}\big)_0\Big)^{N_{B\cap L_P}^o} \xlongrightarrow{\sim}   \Big(\cC^{(n)}(N_{\overline{P}}^o,E)\widehat{\otimes}_E \cC^{\infty,(n)}(N_{\overline{B}\cap L_P}^o,E)\widehat{\otimes}_E\cC^{(n)}(Z_{L_P}^o,E)\widehat{\otimes}_E \cC^{\infty,(n)}(T_P^o,E)\Big)^{\oplus r} \\
  \xlongrightarrow{\sim} \Big(\cC^{(n)}(N_{\overline{P}}^o,E)\widehat{\otimes}_E \cC^{\infty,(n)}(N_{\overline{B}\cap L_P}^o,E)\Big)^{\oplus r} \widehat{\otimes}_E\cC^{(n)}(Z_{L_P}^o,E)\widehat{\otimes}_E \cC^{\infty,(n)}(T_P^o,E)
\end{multline}
Note the strong dual of $\cC^{(n)}(Z_{L_P}^o,E)\widehat{\otimes}_E \cC^{\infty,(n)}(T_P^o,E)$ is $\cD_{<r_n}(Z_{L_P}^o,E)\widehat{\otimes}_E \cD_{<r_n}^{\infty}(T_P^o,E)$. Put thus
\begin{equation*}M_n:=\Big( \Big(\big(\big(\Pi_H^{(n)}\big)^{N_P^o}\big)_0\Big)^{N_{B\cap L_P}^o}\Big)^{\vee}_b \widehat{\otimes}_{\cD_{<r_n}(Z_{L_P}^o,E)\widehat{\otimes}_E \cD_{<r_n}^{\infty}(T_P^o,E)} A_n.
\end{equation*}
Indeed, since $\Big( \Big(\big(\big(\Pi_H^{(n)}\big)^{N_P^o}\big)_0\Big)^{N_{B\cap L_P}^o}\Big)^{\vee}_b$ is obviously an orthonormalisable $\cD_{<r_n}(Z_{L_P}^o,E)\widehat{\otimes}_E \cD_{<r_n}^{\infty}(T_P^o,E)$-module, we see $M_n$ is an orthonormalisable $A_n$-module. The existence of the maps $\alpha_n$, $\beta_n$ follows by the same arguments as in the proof of \cite[Prop.5.3]{BHS1}.
\end{proof}
Let $\chi_0$ be a smooth character of $T_P^o$ over $E$, let $N(\chi_0)\in \Z_{\geq 1}$ such that $\chi_0\in \cC^{(N(\chi_0))}(T_P^o,E)$. Note $\chi_0$ corresponds to a maximal ideal of $\cD^{\infty}_{r_n}(T_P^o,E)$, and thus induces a projection $\cD^{\infty}_{r_n}(T_P^o,E)\twoheadrightarrow E$ for all $n\geq N(\chi_0)$. Put $M_n^{T_P^o=\chi_0}:=M_n \otimes_{\cD^{\infty}_{r_n}(T_P^o,E),\chi_0} E$, $B_n:=\cD_{r_n}(Z_{L_P}^o,E)$, thus $M_n^{T_P^o=\chi_0}$ is an orthonormalisable $B_n$-module. Note the $z_n$-action on $M_n$ induces a compact $B_n$-linear action $z_n$-action on $M_n^{T_P^o=\chi_0}$. By the above proposition, one has
\begin{corollary}\label{cor: gln-ose}
  One has an isomorphism
  \begin{equation*}
    \Big(\big(\big(V^{N_P^o} \otimes_E \sL_P(\lambda)'\big)_0\big)^{N_{B\cap L_P}^o,T_P^o=\chi_0}\Big)^{\vee} \xlongrightarrow{\sim} \varprojlim_{n\geq N(\chi_0)} M_n^{T_P^o=\chi_0},
  \end{equation*}
  and the following diagram commutes
  \begin{equation*}
  \begindc{\commdiag}[40]
  \obj(0,1)[a]{$ \Big(\big(\big(V^{N_P^o} \otimes_E \sL_P(\lambda)'\big)_0\big)^{N_{B\cap L_P}^o,T_P^o=\chi_0}\Big)^{\vee}$}
  \obj(3,1)[b]{$\cdots$}
  \obj(4,1)[c]{$M_{n+1}^{T_P^o=\chi_0}$}
  \obj(6,1)[d]{$M_{n+1}^{T_P^o=\chi_0}\otimes_{B_{n+1}} B_n$}
  \obj(8,1)[e]{$M_n^{T_P^o=\chi_0}$}
    \obj(0,0)[a']{$ \Big(\big(\big(V^{N_P^o} \otimes_E \sL_P(\lambda)'\big)_0\big)^{N_{B\cap L_P}^o,T_P^o=\chi_0}\Big)^{\vee}$}
  \obj(3,0)[b']{$\cdots$}
  \obj(4,0)[c']{$M_{n+1}^{T_P^o=\chi_0}$}
  \obj(6,0)[d']{$M_{n+1}^{T_P^o=\chi_0}\otimes_{B_{n+1}} B_n$}
  \obj(8,0)[e']{$M_n^{T_P^o=\chi_0}$}
  \mor{a}{b}{}
  \mor{b}{c}{}
  \mor{c}{d}{}
  \mor{d}{e}{$\beta_n$}
    \mor{a'}{b'}{}
  \mor{b'}{c'}{}
  \mor{c'}{d'}{}
  \mor{d'}{e'}{$\beta_n$}
  \mor{a}{a'}{$\pi_z$}
  \mor{c}{c'}{$z_{n+1}$}
  \mor{d}{d'}{$z_{n+1}\otimes \id$}[\atright,\solidarrow]
  \mor{e}{e'}{$z_n$}
  \mor{e}{d'}{$\alpha_n$}
  \enddc.
\end{equation*}
\end{corollary}
Keep the above notation, put
\begin{equation*}
  J_ {B,(P,\lambda_0)}^{T_P^o,\chi_0}(V):=\big(\big(V^{N_P^o} \otimes_E \sL_P(\lambda)'\big)_0\big)^{N_{B\cap L_P}^o,T_P^o=\chi_0} \otimes_E \chi_{\lambda},
\end{equation*}
which is thus a closed subrepresentation (hence also essentially admissible) of $J_{B,(P,\lambda_0)}(V)$.

For a topologically finitely generated abelian group $Z$, denote by $\widehat{Z}$ the rigid space over $E$ parameterizing locally analytic characters of $Z$. By \cite[Prop.6.4.6]{Em04}, if $Z$ is moreover compact, one has a natural isomorphism $\co(\widehat{Z}) \xlongrightarrow{\sim} \cD(Z,E)$ (where for a rigid analytic space $\fX$, we use $\co(\fX)$ to denote the global sections on $\fX$). By definition (\cite[Def.6.4.9]{Em04}), there exists an equivalence of categories of the category of coadmissible $\co(\widehat{Z})$-modules and that of essentially admissible locally analytic representations of $Z$.

In particular, the strong dual of $J_{B,(P,\lambda_0)}(V)$ \Big(resp. $J_{B,(P,\lambda_0)}^{T_P^o,\chi_0}(V)$\Big) corresponds to a coherent sheaf $\cM_{\lambda_0}(V)$ \big(resp. $\cM_{\lambda_0}^{\chi_0}(V)$\big) over $\widehat{T}:=\widehat{T(\Q_p)}$, with
\begin{equation*}
  \cM_{\lambda_0}(V)\big(\widehat{T}\big)\xlongrightarrow{\sim} J_{B,(P,\lambda_0)}(V)^{\vee}_b\end{equation*}
  \begin{equation*}\text{\Big(resp. }\cM_{\lambda_0}^{\chi_0}(V)\big(\widehat{T}\big)\xlongrightarrow{\sim} J_{B,(P,\lambda_0)}^{T_P^o,\chi_0}(V)^{\vee}_b\Big).
\end{equation*}The character $\chi_{\lambda}$ induces an isomorphism of rigid spaces:
\begin{equation}\label{equ: gln-Toi}
  \chi_{\lambda}: \widehat{T} \lra \widehat{T}, \ \chi\mapsto \chi\chi_{\lambda}.
\end{equation}
One easily sees the coherent sheaf $\chi_{\lambda}^* \big(\cM_{\lambda_0}(V)\big)$ \big(resp. $\chi_{\lambda}^*\big(\cM_{\lambda_0}^{\chi_0}(V)\big)$\big) corresponds to the $T(\Q_p)$-representation $\big(\big(V^{N_P^o} \otimes_E \sL_P(\lambda)'\big)_0\big)^{N_{B\cap L_P}^o}_{\cfs}$ \Big(resp. $\big(\big(V^{N_P^o} \otimes_E \sL_P(\lambda)'\big)_0\big)^{N_{B\cap L_P}^o,T_P^o=\chi_0}_{\cfs}$\Big). One has
\begin{equation*}
\kappa:  \widehat{T} \twoheadlongrightarrow \widehat{T^o} \times \bG_m \xlongrightarrow{\sim} \widehat{T_P^o} \times \widehat{Z_{L_P}^o} \times \bG_m
\end{equation*}
where the first projection maps $\chi$ to $\big(\chi|_{T^o},\chi(z)\big)$ for any $\chi \in \widehat{T}(\overline{E})$. The character $\chi_0$ corresponds to an $E$-point of $\widehat{T_P^o}$. The affinoids $\{\Spm B_n\}_{n\in \Z_{\geq 1}}$ form an admissible covering of $\widehat{Z_{L_P}^o}$. We view $M_n^{\chi_0}$ as a $B_n[X]$ module with $X$ acting on $M_n$ via $z_n$, and put $M_{n,\cfs}^{\chi_0}:=M_n^{\chi_0}\widehat{\otimes}_{E[X]} E\{\{X,X^{-1}\}\}$. By \cite[Prop.2.2.6]{Em11}, $M_{n,\cfs}^{\chi_0}$ is a coadmissible $B_n\{\{X,X^{-1}\}\}$-module, and thus corresponds to a coherent sheaf $\cM_{n}^{\chi_0}$ over $\Spm B_n\times \bG_m$. By \cite[Prop.2.1.9]{Em11}, one has $M_{n+1,\cfs}^{\chi_0} \widehat{\otimes}_{B_{n+1}\{\{X,X^{-1}\}\}} B_n\{\{X,X^{-1}\}\}\xlongrightarrow{\sim} M_{n,\cfs}^{\chi_0}$. Thus $\{\cM_n^{\chi_0}\}_{n}$ glues to a coherent sheaf $\cM^{\chi_0}$ over $\widehat{T^o} \times \bG_m$. By the isomorphism \begin{equation}\label{equ: gln-sep}
  \varprojlim_{n\geq N(\chi_0)} M_{n,\cfs}^{\chi_0} \xlongrightarrow{\sim} \big(\big(V^{N_P^o} \otimes_E \sL_P(\lambda)'\big)_0\big)^{N_{B\cap L_P}^o,T_P^o=\chi_0}_{\cfs},
\end{equation}
we see  $\kappa_* \big(\chi_{\lambda}^*\big(\cM_{\lambda_0}^{\chi_0}(V)\big)\big)$  is a coherent sheaf over $\widehat{T_P^o} \times \widehat{Z_{L_P}^o} \times \bG_m$ \big(since its global section is no other than the module in (\ref{equ: gln-sep})\big), supported on the closed subspace $\widehat{Z_{L_P}^o}\times \bG_m\hookrightarrow \widehat{T_P^o} \times \widehat{Z_{L_P}^o} \times \bG_m$, $(\chi,x)\mapsto (\chi_0, \chi, x)$ (and we use $\chi_0$ to denote this closed embedding), moreover, $\chi_0^* \kappa_* \big(\chi_{\lambda}^*\big(\cM_{\lambda_0}^{\chi_0}(V)\big)\big) \xrightarrow{\sim} \cM^{\chi_0}$.

Recall for a coherent sheaf $\cM$ over a rigid analytic space $\fX$, the support $\Supp(\cM)$ of $\cM$ is defined to be the closed rigid subspace of $\fX$ such that for any affinoid open $\Spm A$ of $\fX$, $\Supp(\cM)\times_{\fX} \Spm A\cong \Spm(A/I)$ where $I:=\{a \in A\ |\ am=0, \ \forall m\in \cM(\Spm(A))\}$.
Denote by $F_n(X)\in B_n\{\{X,X^{-1}\}\}$ the characteristic power series of the compact operator  $z_n$ on $M_n^{T_P^o=\chi_0}$, as in \cite[Lem.3.10]{BHS1} (see also \cite[Prop.5.A.6]{Ding}), one can prove the support of the coherent sheaf $\cM_n^{\chi_0}$ on $B_{n}\{\{X,X^{-1}\}\}$ is $\Spm F_n(X^{-1})$. By \cite[Prop.6.4.2]{Che}, $\dim \Supp \cM^{\chi_0}=\dim Z_{L_P}^o$. Moreover, by \cite[Prop.A5.8]{Cole97}, there exists an admissible covering $\{\cU_{n,i}\}_i$ of affinoid opens of $\Spm F_n(X^{-1})$, such that the image of $\cU_{n,i}$ via the projection $\Spm B_n \times \bG_m\twoheadrightarrow \Spm B_n$ is an affinoid open, denoted by $\Spm B_{n,i}$, in $\Spm B_n$, and the sections of $\cM_n^{\chi_0}(\cU_{n,i})$ form a finite projective $B_{n,i}$-module. Since (\ref{equ: gln-Toi}) is an isomorphism, we deduce from the above discussion:
\begin{corollary}\label{cor: gln-0iv}
  (1) The support of $\cM_{\lambda_0}^{\chi_0}(V)$ is equidimensional of dimension $\dim Z_{L_P}^o$.

  (2) There exists an admissible covering $\{\cU_i\}$ of $\Supp \cM_{\lambda_0}^{\chi_0}(V)$, such that
  \begin{itemize}
    \item the image $\cV_i$ of $\cU_i$ via the composition $\Supp \cM_{\lambda_0}^{\chi_0}(V)\hookrightarrow \widehat{T}\ra \widehat{Z_{L_P}^o}$ is an affinoid open in $\widehat{Z_{L_P}^o}$,
    \item the sections $\cM_{\lambda_0}^{\chi_0}(V)$ over $\cU_i$ form a finite projective $\co(\cV_i)$-module.
  \end{itemize}
\end{corollary}

\section{Eigenvarieties and closed subspaces}
\subsection{Notations and preliminaries}\label{sec: gln-3.1}We fix embeddings $\iota_{\infty}: \overline{\Q} \hookrightarrow \bC$, and $\iota_{\infty}: \overline{\Q} \hookrightarrow \overline{\Q_p}$.
Let $F^+$ be a totally real number field, $F$ a quadratic imaginary extension of $F^+$, denote by $c$ the unique non-trivial element of $\Gal(F/F^+)$. We suppose $p$ is split in $F$, thus $p$ is split in $F^+$, and for any place $v$ of $F^+$ above $p$, $v$ is split in $F$. Denote by $\Sigma_p$ the places of $F^+$ above $p$.

Let  $G$ be a definite quasi-split unitary group over $F^+$ associated to $F/F^+$, thus $G\times_{F^+} F\xrightarrow{\sim} \GL_n/F$, and $G(F^+\otimes_{\Q} \bR)$ is compact.   Note one has an isomorphism $G(F^+ \otimes_{\Q} \Q_p)\xrightarrow{\sim} \prod_{v\in \Sigma_p} G(F^+_v)\xrightarrow{\sim}  \prod_{v\in \Sigma_p} \GL_n(\Q_p)$ (where the last isomorphism depends on the choice of the places of $F$ above each $v\in \Sigma_p$).


Let $U^p$ be an open compact subgroup of $G(\bA_{F^+}^{p,\infty})$ with the form $U^p=\prod_{v\nmid p} U_v$, put
\begin{equation*}
  \widehat{S}(U^p,E):=\Big\{f: G(F^+) \setminus G(\bA_{F^+}^{\infty})/U^p\ra E\ |\ f \text{ is continuous}\Big\}.
\end{equation*}
Since $G(F^+\otimes_{\Q} \bR)$ is compact, $G(F^+)\setminus G(\bA_{F^+}^{\infty})/U^p$ is a profinite set. We see $\widehat{S}(U^p,E)$ is a Banach space over $E$ with the norm defined by the (completed) $\co_E$-lattice
\begin{equation*}\widehat{S}(U^p,\co_E):=\Big\{f: G(F^+)\setminus G(\bA_{F^+}^{\infty})/U^p \ra \co_E\ |\ f \text{ is continuous}\Big\}.\end{equation*}
Moreover, $\widehat{S}(U^p,E)$ is equipped with a continuous action of $G(F^+\otimes_{\Q}\Q_p)$ given by $(gf)(g')=f(g'g)$ for $f\in \widehat{S}(U^p,E)$, $g\in G(F^+\otimes_{\Q} \Q_p)$, $g'\in G(\bA_{F^+}^{\infty})$. The lattice $\widehat{S}(U^p,\co_E)$ is stable by this action, thus the Banach representation $\widehat{S}(U^p,E)$ of $G(F^+\otimes_{\Q} \Q_p)$ is unitary.
\begin{lemma}\label{lem: gln-htn}Let $H$ be a compact open subgroup of $G(F^+\otimes_{\Q} \Q_p)$ such that $G(F^+)\cap (U^pH)=\{1\}$, then there exists $r\in \Z_{\geq 1}$ such that
  \begin{equation*}
    \widehat{S}(U^p,E)|_H \xlongrightarrow{\sim} \cC(H,E)^{\oplus r}.
  \end{equation*}
  Thus, $\widehat{S}(U^p,E)$ is a unitary admissible Banach representation of $G(F^+\otimes_{\Q} \Q_p)$ over $E$.
\end{lemma}
\begin{proof}
Let $S\subseteq G(\bA_{F^+}^{\infty})$ be a finite representative set of the finite set $G(F^+)\setminus G(\bA_{F^+}^{\infty})/\big(U^pH\big)$. We have a $H$-invariant decomposition
  \begin{equation*}
    \sqcup_{s\in S} \big(sH\big) \xlongrightarrow{\sim} G(F^+)\setminus G(\bA_{F^+}^{\infty})/U^p,
  \end{equation*}
  where any element of $H$ is viewed as an element of $G(\bA_{F^+}^{\infty})$ via $G(F^+\otimes_{\Q} \Q_p)\hookrightarrow G(\bA_{F^+}^{\infty})$. Indeed, the surjectivity follows from the fact $S$ is a representative set of $G(F^+)\setminus G(\bA_{F^+}^{\infty})/(U^pH)$, while the injectivity is from the fact $G(F^+) \cap (U^pH)=\{1\}$. The lemma follows.
\end{proof}


Let $S(U^p)$ be the set of primes $v$ of $F^+$ satisfying
\begin{itemize}
  \item $v\nmid p$, and $v$ is split in $F$;
  \item $U_v$ is a maximal compact open subgroup of $G(F^+_v)$.
\end{itemize}Let $\cH^p_0$ be the spherical Hecke algebra $\co_E\big[\prod_{v\in S(U^p)} G(F^+_v)//\prod_{v\in  S(U^p)} U_v\big]$. Indeed, for $v\in S(U^p)$, let $\tilde{v}$ be a finite place above $v$, which would induce an isomorphism $\iota_{G,\tilde{v}}: G(F^+_v) \xrightarrow{\sim} \GL_n(F_{\tilde{v}})$.
 For $1\leq i \leq n$, let
\begin{equation*}
  T_{\tilde{v}}^{(i)}:=\Big[U_{v} \iota_{G,w}^{-1}\begin{pmatrix}
   \text{\textbf{1}}_{n-i} & 0 \\ 0 & \text{$\varpi_{\tilde{v}}\cdot$ \textbf{1}}_{i}
  \end{pmatrix}U_{\ell}\Big],
\end{equation*}
where $\varpi_{\tilde{v}}$ is a uniformizer of $F_{\tilde{v}}$. Then the Hecke algebra $\co_E[G(F^+_v)//U_v]$ is the $\co_E$-polynomial algebra generated by
$T_{\tilde{v}}^{(i)}$ for $1\leq i \leq n$. Moreover, if we denote by $\tilde{v}^c$ the another place over $v$, then $T_{\tilde{v}^c}^{(i)}=(T_{\tilde{v}}^{(n)})^{-1} T_{\tilde{v}}^{(n-i)}$. Let $S$ be a finite set of places of $F^+$ containing the places $v$ where $U_v$ is ramified, such that $S\cap \Sigma_p=\emptyset$, $S\cap S(U^p)=\emptyset$ and $U_v$ is maximal hyperspecial for $v\notin S \cup \Sigma_p$. Let $\cH^{S,p}$ be the commutative spherical Hecke algebra $\co_E\big[\prod_{v\notin S\cup \Sigma_p} G(F^+_v)//\prod_{v\notin S\cup \Sigma_p} U_v\big]$, thus $\cH^p_0\subset \cH^{S,p}$. Note $\cH^{S,p}$ acts naturally on $\widehat{S}(U^p,E)$, and the action commutes with $G(F^+ \otimes_{\Q} \Q_p)$.

Recall the automorphic representations of $G(\bA_{F^+})$ are the irreducible constituants of the $\bC$-vector space of functions $f: G(F^+)\backslash G(\bA_{F^+})\ra \bC$, which are
\begin{itemize}
  \item $\cC^{\infty}$ when restricted to $G(F^+\otimes_{\Q} \bR)$,
  \item locally constant when restricted to $G(\bA_{F^+}^{\infty})$,
  \item $G(F^+\otimes_{\Q} \bR)$-finite,
\end{itemize}
where $G(\bA_{F^+})$ acts on this space via right translation. An automorphic representation $\pi$ is isomorphic to $\pi_{\infty}\otimes_{\bC} \pi^{\infty}$ where $\pi_{\infty}=W_{\infty}$ is an irreducible algebraic representation of $G(F^+\otimes_{\Q} \bR)$ over $\bC$ and $\pi^{\infty}\cong \Hom_{G(F^+\otimes_{\Q}\bR)}(W_{\infty}, \pi)\cong \otimes'_{v}\pi_v$ is an irreducible smooth representation of $G(\bA_{F^+}^{\infty})$. The algebraic representation $W_{\infty}$ is defined over $\overline{\Q}$ via $\iota_{\infty}$, and we denote by $W_p$ its base change to $\overline{\Q_p}$, which is thus an irreducible algebraic representation of $G(F^+\otimes_{\Q} \Q_p)$ over $\overline{\Q_p}$. Via the decomposition $G(F^+\otimes_{\Q} \Q_p)\xrightarrow{\sim} \prod_{v\in \Sigma_p} G(F^+_v)$, one has $W_p\cong \otimes_{v\in \Sigma_p} W_v$ where $W_v$ is an irreducible algebraic representation of $G(F^+_v)$. One can also prove $\pi^{\infty}$ is defined over a number field via $\iota_{\infty}$ (e.g. see \cite[\S 6.2.3]{BCh}). Denote by $\pi^{\infty,p}:=\otimes'_{v\nmid p} \pi_v$, thus $\pi \cong \pi^{\infty,p}\otimes_{\overline{\Q}} \pi_p$. Let $m(\pi)\in \Z_{\geq 1}$ be the multiplicity of $\pi$ in the space of functions as above.
\begin{proposition}[$\text{\cite[Prop.5.1]{Br13II}}$]\label{prop: gln-stm}
One has a $G(F^+ \otimes_{\Q} \Q_p)  \times \cH^{S,p}$-invariant isomorphism
\begin{equation*}
  \widehat{S}(U^p,E)_{\lalg} \otimes_E \overline{\Q_p} \cong \bigoplus_{\pi}\Big((\pi^{\infty,p})^{U^p} \otimes_{\overline{\Q}}(\pi_p \otimes_{\overline{\Q}} W_p)\Big)^{\oplus m(\pi)},
\end{equation*}
where   $\widehat{S}(U^p,E)_{\lalg}$ denotes the locally algebraic subrepresentation of $\widehat{S}(U^p,E)$, $\pi\cong \pi_{\infty}\otimes_{\bC} \pi^{\infty}$ runs through the automorphic representation of $G(\bA_{F^+})$ and $W_p$ is associated to $\pi_{\infty}$ as above.
\end{proposition}

We fix a place $u$ of $F^+$ above $p$, and $\tilde{u}|u$, thus we fix an isomorphism $i_{G,\tilde{u}}: G(F^+_u)\xrightarrow{\sim} \GL_n(\Q_p)$. Let $W_p^u$ be an irreducible algebraic representation of $\prod_{v|p, v\neq u} G(F^+_v)$ over $E$, $U_p^u=\prod_{v|p,v\neq u} U_v$ be a maximal compact open subgroup of $\prod_{v|p, v\neq u}G(F^+_v)$, and put $U^u:=U^pU_p^u$. Put $\widehat{S}(U^u,W_p^u):=\big(\widehat{S}(U^p,E)\otimes_E W_p^u\big)^{U_p^u}$, which is an admissible unitary Banach representation of $G(F^+_u)$ over $E$, and is equipped with an actin of $\cH^{S,u}$, the $\co_E$-algebra generated by $\cH^{S,p}$ and the spherical Hecke algebra $\co_E\big[G(F^+_v)//U_v\big]$ for $v|p$, $v\neq u$. Moreover, the action of $\cH^{S,u}$ commutes with that of $\GL_n(\Q_p)$. By Lem.\ref{lem: gln-htn}, one has
 \begin{lemma}
   Let $H$ be an  open compact subgroup of $\GL_n(\Q_p)$ \big(viewed as a subgroup of $G(F_u^+)$ via $i_{G,\tilde{u}}$\big) such that $\big(HU^u\big) \cap G(F^+)=1$, then there exists $r\in \Z_{\geq 1}$ such that
   \begin{equation*}
     \widehat{S}(U^u,W_p^u)|_H \xlongrightarrow{\sim} \cC(H,E)^{\oplus r}.
   \end{equation*}
 \end{lemma}
 \begin{proof}Applying Lem.\ref{lem: gln-htn} to the group $H\times U_p^u$, one gets
 \begin{equation*}\widehat{S}(U^p,E)|_{H\times U_p^u}\cong \cC(H\times U^u_p,E)^{\oplus r} \cong \cC(H,E)^{\oplus r} \widehat{\otimes}_E \cC(U_p^u,E).\end{equation*}
By definition, one has
 \begin{equation*}
   \widehat{S}(U^u,W_p^u)|_{H} \cong \cC(H,E)^{\oplus r} \widehat{\otimes}_E \big(\cC(U_p^u,E) \otimes_E W_p^u\big)^{U_p^u},
 \end{equation*}
 since $\big(\cC(U_p^u,E) \otimes_E W_p^u\big)^{U_p}$ is a finite dimensional $E$-vector space, the lemma follows.
 \end{proof}




Let $B$ be the Borel subgroup of $\GL_n$ of upper triangular matrices, $\Phi$ the root system of $\GL_n$, $\Delta$ the set of simple roots with respect to $B$ and $\Phi^+$ (resp. $\Phi^-$) the set of positive (resp. negative) roots, and $\sW\cong S_{n}$ the Weyl group which is generated by the simple reflections $s_{\alpha}$ for all $\alpha\in \Delta$. Each subset $I\subset \Delta$ defines a root system $\Phi_I \subset \Phi$ with positive roots $\Phi_I^+$, negative roots $\Phi_I^-$, and Weyl group $\sW_I \subset \sW$ generated by $s_{\alpha}$ for $\alpha\in I$ (put $\sW_{\emptyset}=\{1\}$). Let $P_I$ be the parabolic subgroup associated to $\Delta_I$, denote by $L_{I}$ the Levi subgroup of $P_I$, $N_I$ the nilpotent radical of $P_I$, $\overline{P}_I$ the opposite parabolic of $P_I$, $\overline{N}_I$ the nilpotent radical of $\overline{P}_I$. Thus $B=P_{\emptyset}$, $G=P_{\Delta}$, put $T:=L_{\emptyset}$, $N:=N_{\emptyset}$ etc. Denote by $\ug$, $\ub$, $\fp_I$, $\fl_I$, $\fn_I$, $\overline{\fp}_I$, $\overline{\fn}_I$, $\ft$, $\fn$ the associated Lie algebras of $\GL_n$, $B$, $P_I$, $L_I$, $N_I$, $\overline{P}_I$, $\overline{N}_I$, $T$, $N$ respectively.

Let $\lambda\in \ft^*:=\Hom_{\Q_p}(\ft,E)$ be a weight, thus there exist $k_{\lambda,i}\in E$ for $i=1,\cdots, n$ such that $\lambda\big((x_i)_{i=1,\cdots,n}\big)=\sum_{i=1}^n k_{\lambda,i} x_i$. Moreover, $\lambda$ is dominant if and only if $k_{\lambda,i}-k_{\lambda,i+1}\in \Z_{\geq 0}$ for $i=1,\cdots, n-1$. We enumerate the simple roots with $\{1,\cdots,n-1\}$ such that $\alpha_j$  gives the weight $(x_i)_{i=1,\cdots, n}\mapsto x_j-x_{j+1}$, and fix hence a bijection $\Delta \xrightarrow{\sim} \{1,\cdots,n-1\}$. Any subset $I$ of $\Delta$ can be viewed as a subset of $\{1,\cdots, n-1\}$,  and a weight $\lambda$ is $P_I$-dominant if and only if $k_{\lambda,i}-k_{\lambda,i+1}\in \Z_{\geq 0}$ for all $i\in I$. A weight $\lambda$ is called \emph{integral} if $k_{\lambda,i}\in \Z$ for all $i=1,\cdots, n$.


Let $I\subset \Delta$, by the highest weight theory, for any $P_I$-dominant integral weight $\lambda$, there exists a unique irreducible finite dimensional algebraic representation $\sL_I(\lambda)$ of $L_I$ with highest weight $\lambda$. This gives an one-to-one bijection between the irreducible finite dimensional algebraic representations of $L_I$ and the $I$-dominant integral weights. Put $\sL(\lambda):=\sL_{\Delta}(\lambda)$. For general $\lambda$, denote by $\sL(\lambda)\in \co^{\overline{\ub}}$ to be the unique simple quotient of the Verma module $\text{U}(\ug) \otimes_{\text{U}(\overline{\ub})} \lambda$.

The Weyl group acts naturally on $\ft^*$ by $s_{\alpha}(\lambda)=\lambda-\langle \lambda, \alpha\rangle \alpha$ for $\alpha\in \Delta$. We would use frequently the \emph{dot action} given by $w \cdot \lambda:=w(\lambda+\frac{1}{2}\sum_{\alpha\in \Phi^+} \alpha)-\frac{1}{2}\sum_{\alpha\in \Phi^+}\alpha$. Note that $s_{\alpha}\cdot \lambda=s_{\alpha}(\lambda)-\alpha$ for $\alpha\in \Delta$ (since $s_{\alpha}$ permutes the set $\Phi^+\setminus \{\alpha\}$ and sends $\alpha$ to $-\alpha$). And we have $k_{w\cdot \lambda,i}=k_{\lambda,w^{-1}(i)}-(w^{-1}(i)-i)$.

For a locally analytic character $\chi$ of $T(\Q_p)$ over $E$, denote by $\wt(\chi)\in \ft^*$ the corresponding weight. 
The character $\chi$ is called locally algebraic (resp. dominant) if $\wt(\chi)$ is integral (resp. dominant). For an integral weight $\lambda$, denote by $\chi_{\lambda}:=z^{k_{\lambda,1}} \otimes \cdots \otimes z^{k_{\lambda,n}}$.

For $1\leq i \leq n$, denote by $\beta_i$ the cocharacter $\bG_m \ra T$, $x\mapsto (\underbrace{x, \cdots, x}_i, 1,\cdots, 1)$.

Let $V$ be an $E$-vector space equipped with an $E$-linear action of $A$ (with $A$ a set of operators), $\chi$ a system of eigenvalues of $A$, denote by $V^{A=\chi}$ the $\chi$-eigenspace, $V[A=\chi]$ the generalized $\chi$-eigenspace
\subsection{Eigenvarieties}  Let $m\in \Z_{>0}$ such that $\{x^i=1\}\cap (1+p^m \Z_p) =\{1\}$ for $i=1\cdots, n$. Let $H$ be the compact open subgroup generated by $\big\{g\in \SL_n(\Z_p)\ |\ g\equiv 1 \pmod{p^m}\big\}$ and $\big\{g\in Z_{\GL_n}(\Z_p)\ |\ g\equiv 1 \pmod{p^m}\big\}$. For any algebraic subgroup $M$ of $\GL_n$, denote by $M^o:=M(\Q_p)\cap H$. Note  the isomorphism (\ref{equ: gln-omo}) holds for any parabolic subgroup of $\GL_n$. For $I\subseteq \Delta$, denote by $L_I(\Q_p)^+$ the subgroup of $L_I(\Q_p)$ defined as in (\ref{equ: gln-1np}) with respect to $N_I^o$, and put $Z_{L_I}(\Q_p)^+:=Z_{L_I}(\Q_p)\cap L_I(\Q_p)^+$.

We identify $\GL_n(\Q_p)$ and $G(F_u^+)$ via $i_{G,\tilde{u}}$. Consider the locally analytic representation $\widehat{S}(U^u,W_p^u)_{\an}$ of $\GL_n(\Q_p)$, which is admissible and dense in $\widehat{S}(U^u,W_p^u)$ (cf. \cite[Thm.7.1]{ST03}).  Let $I\subsetneq \Delta$, $\lambda_0$ be a dominant weight for $\SL_n$, $\chi_0$ a smooth character of $T_I^o:=H\cap T(\Q_p) \cap  L_I^{\cD}(\Q_p)$, applying the functor $J_{B,(P_I,\lambda_0)}^{T_I^o,\chi_0}(\cdot)$ (cf. (\ref{equ: gln-ivP})), we get an essentially admissible locally analytic representation of $T(\Q_p)$ over $E$:
\begin{equation*}
  J_{B,(P_I,\lambda_0)}^{T_I^o,\chi_0}\big(\widehat{S}(U^u,W_p^u)_{\an}\big)
\end{equation*}
which is equipped with a continuous action of $\cH^{S,u}$ commuting with $T(\Q_p)$. By definition (of essentially admissible locally analytic representations), we associate to $J_{B,(P_I,\lambda_0)}^{T_I^o,\chi_0}\big(\widehat{S}(U^u,W_p^u)\big)$ a coherent sheaf $\cN_{I, \lambda_0}^{\chi_0}(U^u, W_p^u)$ over $\widehat{T}$, equipped moreover with an $\co(\widehat{T})$-linear action of $\cH^{S,u}$. By Emerton's method (\cite[\S 2.3]{Em1}), we can construct an eigenvariety from the triplet $\Big\{\cN_{I,\lambda_0}^{\chi_0}(U^u, W_p^u), \widehat{T}, \cH^{S,u}\Big\}$:
\begin{theorem}There exists a rigid analytic space $\cV_{I,\lambda_0}^{\chi_0}(U^u, W_p^u)$ over $E$, together with a finite morphism of rigid spaces
  \begin{equation*}
    i: \cV_{I,\lambda_0}^{\chi_0}(U^u, W_p^u)\lra \widehat{T}
  \end{equation*}
  and a morphism of $E$-algebras (compatible with $i$) with dense image
  \begin{equation*}
    \co(\widehat{T}) \otimes_{\co_E} \cH^{S,u} \lra \co\big(\cV_{I,\lambda_0}^{\chi_0}(U^u, W_p^u)\big),
  \end{equation*}
  satisfying that
  \begin{enumerate}
    \item a closed point $z$ of $\cV_{I,\lambda_0}^{\chi_0}(U^u, W_p^u)$ is uniquely determined by its image $\chi$ in $\widehat{T}(\overline{E})$ and the induced morphism $\fh: \cH^{S,u}\lra \overline{E}$, called a system of eigenvalues of $\cH^{S,u}$, hence a such $z$ would be  denoted by $(\chi,\fh)$;
    \item for a finite extension $L$ of $E$, $(\chi,\fh)\in \cV_{I,\lambda_0}^{\chi_0}(U^u, W_p^u)(L)$ if and only if the corresponding eigenspace
    \begin{equation*}
     \Big(J_{B,(P_I,\lambda_0)}^{T_I^o,\chi_0}\big(\widehat{S}(U^u,W_p^u)_{\an}\big) \otimes_E L\Big)^{T(\Q_p)=\chi, \cH^{S,u}=\fh}
    \end{equation*}
    is non-zero;
    \item there exists a coherent sheaf  $\cM_{I,\lambda_0}^{\chi_0}(U^u, W_p^u)$ over $\cV_{I,\lambda_0}^{\chi_0}(U^u, W_p^u)$ such that $i_* \cM_{I,\lambda_0}^{\chi_0}(U^u, W_p^u) \cong \cN_{I,\lambda_0}^{\chi_0}(U^u, W_p^u)$ and that for an $L$-point $z=(\chi, \fh)$, the special fiber $z^* \cM_{I,\lambda_0}^{\chi_0}(U^u, W_p^u)$ is naturally dual to the (finite dimensional) $L$-vector space
        \begin{equation*}
         \Big(J_{B,(P_I,\lambda_0)}^{T_I^o,\chi_0}\big(\widehat{S}(U^u,W_p^u)_{\an}\big) \otimes_E L\Big)^{T(\Q_p)=\chi, \cH^{S,u}=\fh}.
        \end{equation*}
  \end{enumerate}
\end{theorem}
\begin{remark} \label{rem: gln-y0ct}Denote by $\chi_{\lambda_0}$ the algebraic character of $T(\Q_p) \cap \SL_n(\Q_p)$ with weight $\lambda_0$, by definition, if $(\chi,\lambda)\in \cV_{I,\lambda_0}^{\chi_0}(U^u, W_p^u)(\overline{E})$, then $\chi\chi_{\lambda_0}^{-1}$ is a smooth character of $T(\Q_p)\cap L_I^{\cD}(\Q_p)$ \big(moreover, $\chi\chi_{\lambda_0}^{-1}|_{T_I^o}=\chi_0$\big). From which we deduce $k_{\wt(\chi),i}-k_{\wt(\chi),i+1}=k_{\lambda_0,i}-k_{\lambda_0,i+1}$ if $i, i+1\in I$ \big(since $\ft\cong \ft_{\SL_n}\oplus \fz$, we view $\lambda_0$ as a dominant weight of $\ft$ which equals $0$ when restricted to $\fz$\big).
\end{remark}
By Lem.\ref{lem: gln-htn} and the results in \S \ref{sec: gln-2.4}, one can construct $\cV_{I,\lambda_0}^{\chi_0}(U^u, W_p^u)$ by using Buzzard's eigenvariety machine (\cite{Bu}). Indeed, let $A_i:=\begin{pmatrix}
  p^i & 0 & \cdots & 0 \\
  0 & p^{i-1} & \cdots & 0 \\
  \vdots &\vdots & \ddots & \vdots\\
  0 & 0 & \cdots & 1
\end{pmatrix}\in \GL_{i+1}(\Q_p)$, $U_p^{(i)}:=\begin{pmatrix}
  A_i & 0 \\
  0 & \text{\textbf{1}}_{n-i-1}
\end{pmatrix}$ for $1\leq i \leq n-1$, $S_p:= p \cdot \text{\textbf{1}}_n$, $R_p \subset T(\Z_p)$ be a (finite) representative set of $T(\Z_p)/T^o$. For $m\in \Z_{\geq 1}$ big enough, let $B_m$ be the affinoid algebra as in \S \ref{sec: gln-2.4}, thus $\{\Spm B_m\}_m$ form an admissible covering of $\widehat{Z_{L_I}^o}$. Let $M_n^{T_I^o=\chi_0}$ be the orthonormalisable $B_m$-modules as in Cor.\ref{cor: gln-ose}, which is equipped with a continuous action of $\cH^S$, the commutative $\co_E$-algebra generated by $\cH^{S,u}$ and $U_p^{(i)}$ for $1\leq i \leq n-1$, $S_p$ and the elements in $R_p$. Moreover, the action of $U_p^{(n-1)}$ corresponds to the operator $z_m$ in Cor.\ref{cor: gln-ose} and thus is compact. We apply results of \cite{Bu} to $\big\{M_m^{T_I^o=\chi_0}, \cH^S, U_p^{(n-1)}\big\}$ for each $m$, and glue them;  the resulted rigid space is just $\cV_{I,\lambda_0}^{\chi_0}(U^u, W_p^u)$. Moreover, the action of $T_I^o$ (via $\chi_0$), $Z_{L_I}^o$ \big(note $T^o \cong T_I^o \times Z_{L_I}^o$\big), $U_p^{(i)}$ for $1\leq i\leq n-1$, $S_p$ and $R_p$ gives rise to the morphism of $\cV_{I,\lambda_0}^{\chi_0}(U^u, W_p^u)$ over $\widehat{T}$. One has thus (cf. \cite[Prop.6.4.2]{Che})
\begin{proposition}\label{prop: gln-0lni}
(1) The rigid analytic space $\cV_{I,\lambda_0}^{\chi_0}(U^u, W_p^u)$ is equidimensional of dimension $n-|I|$.

(2) Let $\kappa$ denote the composition $\cV_{I,\lambda_0}^{\chi_0}(U^u, W_p^u)\ra \widehat{T} \ra \widehat{Z_{L_I}^o}$, thus for any closed point $z$ of  $\cV_{I,\lambda_0}^{\chi_0}(U^u, W_p^u)$ there exists an affinoid open neighborhood $\cU$ of $z$, such that $\kappa(\cU)$ is an affinoid open in $\widehat{Z_{L_I}^o}$ and that $\kappa: \cU \ra \kappa(\cU)$ is finite and surjective when restricted to any irreducible component of $\cU$.
\end{proposition}
Note that if $P_I=B$, one can easily check $J_{B,(P_I,\lambda_0)}^{T_I^o,\chi_0}\big(\widehat{S}(U^u,W_p^u)_{\an}\big)\cong J_B\big(\widehat{S}(U^u,W_p^u)_{\an}\big)$, denote by $\cV(U^u, W_p^u)$ the corresponding eigenvariety, which was well studied in \cite{Che}. For general case,  since $J_{B,(P_I,\lambda_0)}^{T_I^o, \chi_0}\big(\widehat{S}(U^u,W_p^u)_{\an}\big)$ is a closed subrepresentation of $J_B\big(\widehat{S}(U^u,W_p^u)_{\an}\big)$, one has
\begin{proposition}
  One has a natural closed embedding
  \begin{equation}\label{equ: gln-gwo}
    \cV_{I,\lambda_0}^{\chi_0}(U^u, W_p^u)\hooklongrightarrow \cV(U^u,W_p^u), \  (\chi,\fh)\mapsto (\chi,\fh).
  \end{equation}
\end{proposition}
\begin{remark}\label{rem: gln-msa}
  Consider the natural morphism $\cV(U^u,W_p^u)\ra \widehat{T^o} \xrightarrow{\sim} \widehat{T_I^o} \times \widehat{Z_{L_I}^o}$, we view $\chi_0 \chi_{\lambda_0}|_{T_I^o}$ as a closed point of $\widehat{T_I^o}$, and put $\cV_{I,\lambda_0}^{\chi_0}(U^u,W_p^u)':=\cV(U^u,W_p^u)\times_{\widehat{T_I^o}\times \widehat{Z_{L_I}^o}} \big(\{\chi_0 \chi_{\lambda_0}|_{T_I^o}\}\times \widehat{Z_{L_I}^o}\big)$, which is thus a closed rigid subspace of $\cV(U^u,W_p^u)$ containing $\cV_{I,\lambda_0}^{\chi_0}(U^u,W_p^u)$ (cf. Rem.\ref{rem: gln-y0ct}). Let's remark that in general $\cV_{I,\lambda_0}^{\chi_0}(U^u,W_p^u)$ is more subtle than $\cV_{I,\lambda_0}^{\chi_0}(U^u,W_p^u)'$ (see Rem.\ref{rem: gln-gii}, \ref{rem: gln-scr} below).
\end{remark}
\begin{definition}
  Let $L$ be a finite extension of $E$, $z=(\chi,\fh)\in \cV(U^u,W_p^u)(L)$ is called classical, if
  \begin{equation*}
    J_B\big(\widehat{S}(U^u,W_p^u)_{\lalg}\otimes_E L\big)^{T(\Q_p)=\chi,\cH^{S,u}=\fh}\neq 0;
  \end{equation*}
  For a closed point $z$ of $\cV_{I,\lambda_0}^{\chi_0}(U^u, W_p^u)$,  $z$ is called classical, if $z$ is a classical point of $\cV(U^u,W_p^u)$ via the closed embedding (\ref{equ: gln-gwo}). \end{definition}
Note $z=(\chi,\fh)$ being classical implies that $\wt(\chi)$ is dominant. Let $Z_{\cl}$ denote the set of classical points in $\cV(U^u,W_p^u)$, it's known that $Z_{\cl}$ is Zariski-dense in $\cV(U^u,W_p^u)$ and accumulates over the points $z=(\chi,\fh)$ with $\chi$ locally algebraic (cf. \cite[\S 6.4.5]{Che}).
\begin{theorem}\label{thm: gln-vv0i}
  The set $Z_{\cl}\cap \cV_{I,\lambda_0}^{\chi_0}(U^u, W_p^u)\big(\overline{E}\big)$ is Zariski-dense in $\cV_{I,\lambda_0}^{\chi_0}(U^u, W_p^u)$ and accumulates over the points $(\chi,\fh)\in \cV_{I,\lambda_0}^{\chi_0}(U^u, W_p^u)\big(\overline{E}\big)$ with $\chi$ locally  algebraic.
\end{theorem}
\begin{remark}\label{rem: gln-gii}
By Thm.\ref{thm: gln-vv0i},  $\cV_{I,\lambda_0}^{\chi_0}(U^u, W_p^u)_{\red}$ \big(the reduced subspace of $\cV_{I,\lambda_0}^{\chi_0}(U^u,W_p^u)$\big) is actually the  Zariski-closure of the classical points in $\cV_{I,\lambda_0}^{\chi_0}(U^u, W_p^u)'$ (cf. Rem.\ref{rem: gln-msa}).
\end{remark}
\begin{proof}[Proof of Thm.\ref{thm: gln-vv0i}]By Prop.\ref{prop: gln-0lni} (and the discussion above it),  and some standard arguments as in \cite[Prop.6.2.7, 6.4.6]{Che}, Thm.\ref{thm: gln-vv0i} follows from the classicality criterion Thm.\ref{thm: gln-enf} below \big(see also Rem.\ref{rem: gln-k11}, the $\{\wt(\chi_j')\}_{j\in \{1,\cdots,r\}}$ in Rem.\ref{rem: gln-k11} gives the weight of $\chi|_{Z_{L_I}^o}$\big). Note $(\chi,\fh)\mapsto \us\big(\chi\big(\beta_{a_j}(p)\big)\big)$ (see \S \ref{sec: 3.2.1} below for notations) is locally constant on $\cV_{I,\lambda_0}^{\chi_0}(U^u, W_p^u)$.\end{proof}
\subsubsection{A result of classicality} \label{sec: 3.2.1}Let $I\subset \{1,\cdots,n-1\}$ (where we identify the latter set with $\Delta$ as in \S \ref{sec: gln-3.1}), one can get a partition of the ordered set $\{1,\cdots,n\}=S_1\sqcup\cdots \sqcup S_r$ such that $L_I=\GL_{|S_1|}\times \cdots \times \GL_{|S_r|}$. For $j\in \{1,\cdots, r\}$, put $a_j:=|S_1|+\cdots +|S_j|$, and $a_0=0$. One has $\beta_{\alpha_j}(p)\in Z_{L_I}(\Q_p)^+$ for $j\in \{1,\cdots,r\}$. The main result of this section is (compare with  \cite[Prop.4.7.4]{Che})



\begin{theorem}\label{thm: gln-enf} Let $L$ be a finite extension of $E$, $z=(\chi,\fh)$ be an $L$-point of $\cV_{I,\lambda_0}^{\chi_0}(U^u, W_p^u)$ with $\chi$ locally algebraic and dominant, if (where $\us(\cdot)$ denotes the $p$-adic valuation normalized by $\us(p)=1$)
\begin{equation}\label{equ: gln-gab}
  \us\Big(\chi\big(\beta_{a_j}(p)\big)\delta_B^{-1}\big(\beta_{a_j}(p)\big)\Big)\leq k_{\wt(\chi),a_j}-k_{\wt(\chi),a_j+1}+1,\ \forall 1\leq j\leq r-1\end{equation}
 then the point $z$ is classical.
\end{theorem}
\begin{remark}\label{rem: gln-k11}
By Rem.\ref{rem: gln-y0ct}, $k_{\wt(\chi),a_j+1}-k_{\wt(\chi),a_{j+1}}=k_{\lambda_0,a_j+1}-k_{\lambda_0,a_{j+1}}$ for $0\leq j\leq r-1$. 
   It's  easy to see $\wt(\chi)$ is determined by $k_{\wt(\chi),a_j}$ for $j\in \{1,\cdots,r\}$ and $\lambda_0$. Moreover, consider the restriction $\chi|_{Z_{L_I}}=\chi'_1\otimes\cdots \chi'_r$ as a character of $\prod_{j=1}^r \Q_p^{\times}$, with $\chi'_i=\chi_{a_{i-1}+1}\cdots \chi_{a_i}$, thus \begin{equation*}\wt(\chi'_j)=k_{\wt(\chi),a_i} |S_j|+\sum_{i= a_{j-1}+1}^{a_j} (k_{\lambda_0,i}-k_{\lambda_0,a_j}).\end{equation*}
   For $j\in \{1,\cdots,r\}$, denote by $N_j:=\sum_{i= a_{j-1}+1}^{a_j} (k_{\lambda_0,i}-k_{\lambda_0,a_j})$, the numerical criterion in (\ref{equ: gln-gab}) can be reformulated as
   \begin{equation*}
       \us\Big(\chi\big(\beta_{a_j}(p)\big)\delta_B^{-1}\big(\beta_{a_j}(p)\big)\Big)\leq \frac{\wt(\chi_j')-N_j}{|S_j|}-\frac{\wt(\chi_{j+1}')-N_{j+1}}{|S_{j+1}|}+(k_{\lambda_0,a_j+1}-k_{\lambda_0,a_{j+1}})+1
   \end{equation*}
   for $j\in \{1,\cdots, r-1\}$.
\end{remark}
The theorem follows easily from the following proposition
\begin{proposition}\label{prop: gln-ngi}
  Keep the above notation, if (\ref{equ: gln-gab}) holds,
  any vector $v$ in the generalized eigenspace
  \begin{equation*}
    J_{B,(P_I,\lambda_0)}\big(\widehat{S}(U^u,W_p^u)_{\an}\otimes_E L\big)^{T^o=\chi}[T(\Q_p)=\chi,\cH^{S,u}=\fh]
  \end{equation*}
  is locally algebraic, i.e. $v\in J_{B,(P_I,\lambda_0)}\big(\widehat{S}(U^u,W_p^u)_{\lalg}\otimes_E L\big)$.
\end{proposition}
Without loss of generality, we assume $L=E$, since $\chi$ is locally algebraic and has dominant weight, we can write $\chi=\chi_{\wt(\chi)}\psi$ where $\psi$ is a smooth character of $T(\Q_p)$ over $E$. For any non-zero vector
\begin{equation*}
  v\in J_{B,(P_I,\lambda_0)}\big(\widehat{S}(U^u,W_p^u)_{\an}\big)^{T^o=\chi}[T(\Q_p)=\chi,\cH^{S,u}=\fh],
\end{equation*}
the $T(\Q_p)$ subrepresentation generated by $v$ is isomorphic to $\chi_{\wt(\chi)}\otimes_E \pi_{\psi}$ with $\pi_{\psi}$ a finite length smooth representation of $T(\Q_p)$, whose irreducible components are all $\psi$. By the adjunction formula Thm.\ref{thm: gln-ycn}, the injection $\chi_{\wt(\chi)}\otimes_E \pi_{\psi}\hookrightarrow  J_{B,(P_I,\lambda_0)}\big(\widehat{S}(U^u,W_p^u)_{\an}\big)$ induces a non-zero map
\begin{equation}\label{equ: gln-ngug}
  \cF_{\overline{P}_I}^{\GL_n}\Big(\big(\text{U}(\ug)\otimes_{\text{U}(\overline{\fp}_I)} \sL_I(\wt(\chi))'\big)^{\vee}, \big(\Ind_{\overline{B}\cap L_I(\Q_p)}^{L_I(\Q_p)} \pi_\psi\otimes_E\delta_B^{-1}\big)^{\infty}\Big)\lra \widehat{S}(U^p,E)_{\an}.
\end{equation}
Recall any irreducible component of $\big(\text{U}(\ug)\otimes_{\text{U}(\overline{\fp}_I)} \sL_I(\wt(\chi))'\big)^{\vee}\in \co^{\overline{\ub}}$ has the form $\sL(w\cdot (-\wt(\chi)))$ with $w\cdot (-\wt(\chi))$ being $P_I$-dominant \big(recall $-\wt(\chi)$ denotes the highest weight of $\sL(\wt(\chi))'$\big). Denote by $\pi_I:=\big(\Ind_{\overline{B}\cap L_I(\Q_p)}^{L_I(\Q_p)} \psi\delta_B^{-1}\big)^{\infty}$ for simplicity, and note $\pi_I$ has the central character $\psi\delta_B^{-1}$.

Let $w\in \sW$ such that $w\cdot (-\wt(\chi))$ is $P_I$-dominant, let $I_w\supseteq I$ such that $P_{I_w}$ is the maximal parabolic subgroup for $\sL(w\cdot (-\wt(\chi)))$. Let $\pi_w$ be an irreducible component of $\big(\Ind_{P_I\cap L_{I_w}(\Q_p)}^{L_{I_w}(\Q_p)} \pi_I\big)^{\infty}$, which also has the central character $\psi\delta_B^{-1}$. We would show

\begin{lemma}\label{lem: gln-ssn}
Keep the above notation, and suppose the conditions in (\ref{equ: gln-gab}) hold, then the irreducible representation $\cF_{\overline{P}_{I_w}}^{\GL_n}\big(\sL(w\cdot (-\wt(\chi))), \pi_w\big)$ does not admit any invariant lattice. Consequently, $\cF_{\overline{P}_{I_w}}^{\GL_n}\big(\sL(w\cdot (-\wt(\chi))),\pi_w\big)$ can not be a subrepresentation of $\widehat{S}(U^u,W_p^u)_{\an}$ \big(since $\widehat{S}(U^u,W_p^u)$ is unitary\big).
\end{lemma}
\begin{proof}[Proof of Prop.\ref{prop: gln-ngi}]
  By Thm.\ref{thm: gln-pst} and the discussion that precedes  Lem.\ref{lem: gln-ssn}, any irreducible component of  $$\cF_{\overline{P}_I}^{\GL_n}\Big(\big(\text{U}(\ug)\otimes_{\text{U}(\overline{\fp})} \sL_I(\wt(\chi))'\big)^{\vee}, \big(\Ind_{\overline{B}\cap L_I(\Q_p)}^{L_I(\Q_p)} \pi_\psi\otimes_E\delta_B^{-1}\big)^{\infty}\Big)$$
  has the form $\cF_{\overline{P}_{I_w}}^{\GL_n}\big(\sL(w\cdot (-\wt(\chi))),\pi_w\big)$ (with the above notations). By  Lem.\ref{lem: gln-ssn}, we see any non-zero map as in (\ref{equ: gln-ngug}) factors through
  \begin{equation*}
     \cF_{\overline{B}}^{\GL_n}\big(\sL(\wt(\chi))',\pi_{\psi} \otimes_E \delta_B^{-1}\big)\cong \big(\Ind_{\overline{B}}^{\GL_n}\pi_{\psi}\otimes_E \delta_B^{-1}\big)^{\infty}\otimes_E \sL(\wt(\chi)),
  \end{equation*}
  and thus has image in $\widehat{S}(U^u,W_p^u)_{\lalg}$.
\end{proof}
\begin{proof}[Proof of Lem.\ref{lem: gln-ssn}]By \cite[Cor.3.5]{Br13I}, if $\cF_{\overline{P}_w}^{\GL_n}\big(\sL(w\cdot (-\wt(\chi))), \pi_w\big)$ admits an invariant lattice, then
  \begin{equation*}
  \chi_{w\cdot \wt(\chi)}(z)\psi(z)\delta_B^{-1}(z)\in \co_E
\end{equation*}
for all $z\in Z_{L_{I_w}}(\Q_p)^+$. Note $\chi_{w\cdot \wt(\chi)}\psi \delta_B^{-1}=\chi_{w\cdot \wt(\chi)-\wt(\chi)}\chi\delta_B^{-1}$.
By the lemma below \big(applied to $\lambda=\wt(\chi)$, $P=P_{I_w}$\big), there exists $1\leq j_w \leq r$ such that $\beta_{a_{j_w}}(p)\in Z_{L_{I_w}}(\Q_p)^+$ \big(indeed, since $P_I\subseteq P_{I_w}$, for $1\leq a \leq n$, if $\beta_a(p)\in Z_{L_{I_w}}(\Q_p)^+$, one gets $a=a_j$ for some $1\leq j \leq r$\big), and $\langle w\cdot \wt(\chi)-\wt(\chi), \beta_{j_w}\rangle \leq k_{\wt(\chi),a_{j_w}+1}-k_{\wt(\chi),a_{j_w}}-1$. Thus $\us\big(\chi_{w\cdot \wt(\chi)-\wt(\chi)}(\beta_{j_w}(p))\big)\leq k_{\wt(\chi),a_{j_w}+1}-k_{\wt(\chi),a_{j_w}}-1$. Consequently,
\begin{equation*}
  \us\Big(\chi(\beta_{j_w}(p))\delta_B^{-1}(\beta_{j_w}(p))\Big)\geq k_{\wt(\chi),a_{j_w}}-k_{\wt(\chi),a_{j_w}+1}+1,
\end{equation*}
a contradiction with (\ref{equ: gln-gab}).
\end{proof}
\begin{lemma}
  Let $P\supseteq B$ be a parabolic subgroup of $\GL_n$, $\lambda$ be a dominant weight of $\ft$, $w\in \sW$, $w\neq 1$ such that $w\cdot \lambda$ is $P$-dominant, then there exists $1\leq a \leq n$, such that the cocharacter $\beta_{a}$ satisfies $\beta_{a}(p)\in Z_{L_P}(\Q_p)^+$ and $\langle w\cdot \lambda-\lambda, \beta_{a}\rangle\leq k_{\lambda,a+1}-k_{\lambda,a}-1$.
\end{lemma}
\begin{proof}
Consider the partition of the ordered set $\{1,\cdots, n\}=S_{P,1}\sqcup \cdots \sqcup S_{P,s}$ such that $L_P=\GL_{|S_{P,1}|} \times \cdots \times \GL_{|S_{P,s}|}$.
Since $w\cdot \lambda$ is $P$-dominant, $k_{w\cdot \lambda,i}\geq k_{w \cdot \lambda,i+1}$ if $i,i+1\in S_{P,j}$ for some $j\in \{1,\cdots,s\}$. Since $w\neq 1$, there exists $i$ such that $w^{-1}(i)\neq i$. Let $i_0 \in \{1,\cdots, n\}$ be the smallest number such that $w^{-1}(i_0)\neq i_0$ (thus $i_0<w^{-1}(i_0)$), let $j_0\in \{1,\cdots,r\}$ such that $i_0\in S_{P,j_0}$. Note $k_{w\cdot \lambda,i_0}=k_{\lambda,w^{-1}(i_0)}-(w^{-1}(i_0)-i_0)<k_{\lambda,i}$, for $i_0 < i \leq w^{-1}(i_0)$.

If $i\in S_{P,j_0}$, $i\geq i_0$, we claim $w^{-1}(i)\geq w^{-1}(i_0)+i-i_0$. Indeed, if $i_0+1\in S_{P,j_0}$, and if $w^{-1}(i_0+1)<w^{-1}(i_0)$ then $k_{w\cdot \lambda,i_0+1}=k_{\lambda,w^{-1}(i_0+1)}-(w^{-1}(i_0+1)-i_0-1)\geq k_{\lambda,w^{-1}(i_0)}-(w^{-1}(i_0+1)-i_0-1)>k_{w\cdot \lambda,i_0}$ which contradicts to the fact $w\cdot \lambda$ is $P$-dominant. The claim follows then by induction on $i-i_0$.

We has thus  $k_{w\cdot \lambda,i}<k_{\lambda,i}$ for $i\geq i_0$, $i\in S_{P,j_0}$, let $a:=|S_{P,1}|+\cdots+|S_{P,j_0}|$, one has
\begin{multline*}
  \langle w\cdot \lambda-\lambda, \beta_{a}\rangle=\sum_{i\geq i_0,i\in S_{P,j_0}} (k_{w\cdot \lambda,i}-k_{\lambda,i})=\sum_{i\geq i_0, i\in S_{P,j_0}} (k_{\lambda, w^{-1}(i)}-(w^{-1}(i)-i)-k_{\lambda,i})\\ \leq \sum_{i\geq i_0,i\in S_{P,j_0}} (k_{\lambda,w^{-1}(i)}-1-k_{\lambda,i}) \leq k_{\lambda, w^{-1}(a)}-k_{\lambda,a}-1\leq k_{\lambda,a+1}-k_{\lambda,a}-1.
\end{multline*}
%
The lemma follows.
\end{proof}
\subsection{Families of Galois representations}\label{sec: 3.3}
Let $\rho$ be an $n$-dimensional continuous representation  of $\Gal(\overline{F}/F)$ over $E$ satisfying
\begin{itemize}
\item $\rho^c\cong \rho^{\vee}\otimes_E \varepsilon^{1-n}$ where $\rho^c(g):=\rho(cgc)$ for all $g\in \Gal(\overline{F}/F)$, $c$ is the unique non-trivial element in $\Gal(F/F^+)$, $\varepsilon$ denotes the cyclotomic character;
\item $\rho$ is unramified for all but finitely many places of $F$, and $\rho$ is unramified for places of $F$ lying above places in $S(U^p)$.
\end{itemize}Note by Chebotarev density theorem, $\rho$ is determined by $\rho(\Frob_{\tilde{v}})$ for $\tilde{v}|v$, $v\in S(U^p)$. One can associate to $\rho$ a maximal ideal $\fm_{\rho}$ of $\cH^{S,p}_0 \otimes_{\co_E}E$ generated by elements $\Big\{(-1)^j \Norm(\tilde{v})^{\frac{j(j-1)}{2}} T_{\tilde{v}}^{(j)}-a_{\tilde{v}}^{(j)}\Big\}_{j,{\tilde{v}}}$ where $j\in \{1,\cdots, n\}$, ${\tilde{v}}\in S(U^p)$, $\Norm({\tilde{v}})$ is the cardinality of the residue field of $F_{\tilde{v}}$, $X^n +a_{\tilde{v}}^{(1)} X^{n-1}+ \cdots + a_{\tilde{v}}^{(n-1)}X+a_{\tilde{v}}^{(n)}\in E[X]$ is the characteristic polynomial of $\rho(\Frob_{\tilde{v}})$ with $\Frob_{\tilde{v}}$ a geometric Frobenius at ${\tilde{v}}$. Put
\begin{equation}\label{equ: gln-ohr}
  \widehat{\Pi}(\rho):=\widehat{S}(U^u,W_p^u)^{\fm_{\rho}}
\end{equation}
the subspace of $\widehat{S}(U^u,W_p^u)$ annulated by $\fm_{\rho}$, which is thus an admissible unitary Banach representation of $G(F_u^+)\cong \GL_n(\Q_p)$ (via $i_{G,\tilde{u}}$). The Galois representation $\rho$ is called \emph{modular} if $\widehat{\Pi}(\rho)_{\lalg}\neq 0$, in other word, if $\rho$ is associated to certain automorphic representation of $G$ (cf. Prop.\ref{prop: gln-stm}); $\rho$ is called \emph{promodular} if $\widehat{\Pi}(\rho)\neq 0$.

Recall to any closed point $z=(\chi,\fh)$ of $\cV(U^u,W_p^u)$ \big(thus of $\cV_{I,\lambda_0}^{\chi_0}(U^u, W_p^u)$\big), one can associate a continuous absolutely semi-simple representation $\rho_z: \Gal(\overline{\Q}/\cE) \ra \GL_n(k(z))$ where $k(z)$ denotes the residue field at $z$. Suppose moreover \begin{itemize}
\item $\chi$ is locally algebraic,
\item $\wt(\chi)=w\cdot \lambda$ for  certain integral dominant weight $\lambda$;
\end{itemize}
let $(\psi_1,\cdots, \psi_n):=\chi\chi_{\wt(\chi)}^{-1}$, by global triangulation theory \big(\cite{KPX} \cite{Liu} applied to $\cV(U^u,W_p^u)$\big), the restriction $\rho_{z,\tilde{u}}:=\rho_z|_{\Gal(\overline{\Q_p}/F_{\tilde{u}})}$ is trianguline \emph{of parameter $(\delta_1,\cdots, \delta_n)$}, (e.g. see \cite[\S 7]{Br13I}) \Big(i.e. the $(\varphi,\Gamma)$-module $D_{\rig}(\rho_{z,\tilde{u}})$ over $\cR_L$ admits a filtration $\{\Fil^i\}_{i=0,\cdots,n}$ with $\Fil^0 D_{\rig}(\rho_{z,\tilde{u}})=0$, $\Fil^n D_{\rig}(\rho_{z,\tilde{u}})=D_{\rig}(\rho_{z,\tilde{u}})$, and $\gr^i D_{\rig}(\rho_{z,\tilde{u}}):=\Fil^{i}/\Fil^{i-1} D_{\rig}(\rho_{z,\tilde{u}})\cong \cR_L(\delta_i)$ is a rank $1$ $(\varphi,\Gamma)$-module over $\cR_L$ associated to $\delta_i$\Big), where $\delta_i: \Q_p^{\times} \ra L^{\times}$ are continuous characters given by
\begin{equation}\label{equ: gln-sp1}
  (\delta_1,\cdots,\delta_i,\cdots, \delta_n)=\big(\psi_1\unr(p^{n-1}),\cdots, \psi_i\varepsilon^{1-i} \unr(p^{n+1-2i}), \cdots, \psi_n \varepsilon^{1-n} \unr(p^{1-n})\big) \chi_{w(z) \cdot \lambda},
\end{equation}
with $w(z)\in \sW$.

We define an order on $n$-tuples in $\Z^{n}$: $(k_1,\cdots, k_n)\leq (k_1',\cdots, k_n')$ if $\sum_{j=1}^i k_j\leq \sum_{j=1}^i k_j'$ for all $1\leq i \leq n$, and define an action of $\sW\cong S_n$ on $\Z^n$ by $w(k_1,\cdots, k_n)=(k_{w^{-1}(1)},\cdots, k_{w^{-1}(n)})$. Let $\ul{h}(z):=\big(-k_{\lambda,1}, \cdots, -(k_{\lambda,i}+1-i),\cdots, -(k_{\lambda,n}+1-n)\big)$ which are in fact the Hodge-Tate weights of $\rho_{z,\tilde{u}}$ (where we choose the convention that the Hodge-Tate weight of the cyclotomic character $\varepsilon$ is $-1$), and is\emph{ anti-dominant}, i.e. strictly increasing.
\begin{proposition}[$\text{\cite[Prop.9.2]{Br13II}}$]\label{prop: gln-enn}Keep the above notation and assumption, one has $w(\ul{h}(z)) \leq w(z)(\ul{h}(z))$.
\end{proposition}
\begin{remark}
  Note that, conjecturally, one should have  $w \leq w(z)$ by the Bruhat's ordering  (cf. \cite[Rem.9.4 (2)]{Br13II}).
\end{remark}
\begin{definition}[$\text{\cite[Def.2.10]{BHS2}}$]\label{def: gln-fon}
  Let $\rho_p$ be a crystalline representation of $\Gal(\overline{\Q_p}/\Q_p)$ of dimension $n$ over $E$, we say $\rho_p$ is very regular if
 \begin{itemize}
   \item $\rho_p$ has distinct Hodge-Tate weights,
   \item the eigenvalues $\{\phi_1,\cdots, \phi_n\}$of the crystalline Frobenius on $D_{\cris}(\rho_p)$ satisfies $\phi_i\phi_j^{-1}\neq 1,p$, for $i\neq j$,
   \item $\phi_1\phi_2 \cdots \phi_i$ is a simple eigenvalue of the crystalline Frobenius on $\wedge^i_E D_{\cris}(\rho_p)$ for $1\leq i\leq n$.
 \end{itemize}
\end{definition}The following proposition follows from the same argument as in \cite[Thm.4.8, Thm.4.10]{Che11}
\begin{proposition}\label{prop: gln-hlt}Let $z=(\chi,\fh)$ be a classical point of $\cV_{I,\lambda_0}^{\chi_0}(U^u, W_p^u)$ satisfying that $\rho_z$ is absolutely irreducible,  $\rho_{z,\tilde{v}}$ is crystalline and very regular for all $\widetilde{v}|p$. Suppose $n\leq 3$, or $F/F^+$ unramified, $G$ quasi-split at all finite places, $U_v$ maximal hyperspecial at all inert places. If
  \begin{multline}\label{equ: gln-sese}
    J_{B,(P_I,\lambda_0)}\big(\widehat{S}(U^u,W_p^u)_{\lalg}\otimes_E L\big)^{T^o=\chi}[T(\Q_p)=\chi,\cH^{S,u}=\fh] \\ \xlongrightarrow{\sim} J_{B,(P_I,\lambda_0)}\big(\widehat{S}(U^u,W_p^u)_{\an}\otimes_E L\big)^{T^o=\chi}[T(\Q_p)=\chi,\cH^{S,u}=\fh],
  \end{multline}
  then the map $\kappa: \cV_{I,\lambda_0}^{\chi_0}(U^u, W_p^u) \ra \widehat{Z_{L_I}^o}$ is \'etale at $z$.
\end{proposition}
\begin{proof}
Indeed, by the the discussion above Prop.\ref{prop: gln-0lni}, one can reduce to the same situation as in the beginning of the proof of \cite[Thm.4.8]{Che11}. Then Prop.\ref{prop: gln-ngi} ensures there exists a set of classical points satisfying (\ref{equ: gln-sese}) which accumulates over $z$ (as in  \cite[(4.19)]{Che11}). The hypothesis on $z$ gives the property \cite[(4.20)]{Che11}. This proposition then follows from the multiplicity one result as in the proof of \cite[Thm.4.8, Thm.4.10]{Che11}.
\end{proof}
\begin{lemma}\label{lem: gln-nin}
  Let $w,w'\in \sW$, $\ul{h}=(h_1,\cdots, h_n)\in \Z^n$, $h_{i}<h_{i+1}$ for $i=\{1,\cdots, n-1\}$, suppose $w(\ul{h})\geq w'(\ul{h})$, then if $w\in \sW_I$, so is $w'$.
\end{lemma}
\begin{proof}If $w'\notin \sW_I$,
let $i\in \{1,\cdots,n\}$ be the smallest number such that $(w')^{-1}(i)$ and $i$ do not belong to the same partition defined as in the beginning of \S \ref{sec: 3.2.1}, let $i_0\in \{1,\cdots r\}$ such that $i\in S_{i_0}$ (see the beginning of \S \ref{sec: 3.2.1}), let $a:=|S_1|+\cdots+|S_{i_0}|$, one sees easily (where the first equality follows from $w\in \sW_I$)
\begin{equation*}
  \sum_{i=1}^{a} h_{w^{-1}(i)}=\sum_{i=1}^a h_i < \sum_{i=1}^a h_{(w')^{-1}(i)},
\end{equation*}
a contradiction.
\end{proof}
The following theorem generalizes \cite[Thm.4.8,Thm.4.10]{Che11}.
\begin{theorem}\label{thm: gln-tv0}
  Let $z=(\chi,\fh)$ be a classical point of $\cV_{I,\lambda_0}^{\chi_0}(U^u, W_p^u)$ satisfying that $\rho_z$ is absolutely irreducible,  $\rho_{z,\widetilde{v}}$ is crystalline and very regular for all $\widetilde{v}|p$. Suppose $n\leq 3$, or $F/F^+$ unramified, $G$ quasi-split at all finite places, $U_v$ maximal hyperspecial at all inert places.  If $w(z)\in \sW_I$ (cf. (\ref{equ: gln-sp1})), then $\cV_{I,\lambda_0}^{\chi_0}(U^u, W_p^u)$ is \'etale over $\widehat{Z_{L_I}^o}$ at $z$.
\end{theorem}
 \begin{proof}
   By Prop.\ref{prop: gln-hlt}, it's sufficient to show that if $w(z)\in \sW_I$, then any vectors in the generalized eigenspace
   \begin{equation*}
     J_{B,(P_I,\lambda_0)}\big(\widehat{S}(U^u,W_p^u)_{\an}\otimes_E L\big)^{T^o=\chi}[T(\Q_p)=\chi,\cH^{S,u}=\fh]
   \end{equation*}
   is locally algebraic (cf. Prop.\ref{prop: gln-ngi}). And as in the proof of Prop.\ref{prop: gln-ngi} (and we use the notations there), it's sufficient to prove that  $\cF_{\overline{P}_{I_w}}^{\GL_n}\big(\sL(-w\cdot \wt(\chi)), \pi_w\big)$ (see Lem.\ref{lem: gln-ssn}) can not be a subrepresentation of $\widehat{S}(U^p,E)[\cH^{S,u}=\fh]$  if $w\neq 1$ and $w\cdot \wt(\chi)$ is $P_I$-dominant.   Suppose there exist $w\neq 1$, $w\cdot \wt(\chi)$ being $P_I$-dominant and an injection
   \begin{equation*}
     \cF_{\overline{P}_{I_w}}^{\GL_n}\big(\sL(-w\cdot \wt(\chi)), \pi_w\big) \hooklongrightarrow\widehat{S}(U^u,W_p^u)_{\an}[\cH^{S,u}=\fh].
   \end{equation*}
  Applying the Jacquet-Emerton functor $J_B(\cdot)$, by \cite[Cor.3.4]{Br13I}, we get a closed point $z'=(\chi',\fh)\in \cV(U^u,W_p^u)$ with $\chi'=\chi\chi_{{w\cdot \lambda}-\lambda}$ (see also the proof of \cite[Thm.9.3]{Br13II}). By Prop.\ref{prop: gln-enn}, $w(\ul{h}(z'))\leq w(z')(\ul{h}(z'))$, note $w(z')=w(z)$, $\ul{h}(z')=\ul{h}(z)$, and thus by Lem.\ref{lem: gln-nin}  $w\in \sW_I$, which contradicts the fact that $w \cdot \wt(\chi)$ is $P_I$-dominant and $w\neq 1$.
 \end{proof}
Conversely, we have the following result which follows directly from results of Bergdall (\cite{Bergd14})
\begin{theorem}\label{thm: gln-tvI0}
  Let $z=(\chi,\fh)$ be a classical point of $\cV_{I,\lambda_0}^{\chi_0}(U^u, W_p^u)$ satisfying that $\rho_z$ is absolutely irreducible,  and that $\rho_{z,\widetilde{u}}$ is crystalline and very regular. Suppose $w(z) \notin \sW_I$, then $\cV_{I,\lambda_0}^{\chi_0}(U^u, W_p^u)$ is not \'etale over $\widehat{Z_{L_I}^o}$ at $z$.
\end{theorem}
\begin{proof}Denote by $\cV:=\cV_{I,\lambda_0}^{\chi_0}(U^u, W_p^u)$, $\cW:=\widehat{Z_{L_I}^o}$ for simplicity.
  Consider the tangent map $\nabla_z$:  $T_{\cV,z}\ra T_{\widehat{T},\chi}$, and let
  \begin{equation*}X:=\big\{(x_i)\in T_{\widehat{T},\chi}\cong k(z)^{n}\ |\ x_i=x_{i'}, \text{ for } i, i'\in S_j, \ j\in \{1,\cdots,r\}\big\}
   \end{equation*}
We have $\Ima(\nabla_z)\subseteq X$ (e.g. see Rem.\ref{rem: gln-gii}). The natural map $T_{\widehat{T},\chi}\ra T_{\cW,\kappa(z)}$ thus induces an isomorphism $X\xrightarrow{\sim} T_{\cW,\kappa(z)}$. Since $w(z) \notin \sW_I$, there exists $i\in \{1,\cdots, n\}$ such that $(w(z))^{-1}(i)$ and $i$ do not belong to the same partition defined by $I$ (cf. \S \ref{sec: 3.2.1}). By \cite[Thm.B]{Bergd14}, $\Ima(\nabla_z)_i=\Ima(\nabla_z)_{w(z)^{-1}(i)}$, from which we deduce $\nabla_z$ is not surjective onto $X$, the theorem follows.
\end{proof}
\begin{remark}\label{rem: gln-scr}
  The two above theorems also imply that in general the rigid space $\cV_{I,\lambda_0}^{\chi_0}(U^u, W_p^u)'$ (cf. Rem.\ref{rem: gln-msa}) is different from $\cV_{I,\lambda_0}^{\chi_0}(U^u, W_p^u)$. For example, applying Thm.\ref{thm: gln-tvI0} to $\cV(U^u,W_p^u)$, if $z$ is a classical point with $w(z)\neq 1$, then $\cV(U^u,W_p^u)$ is not \'etale over $\widehat{T^o}$ at $z$ with the induced tangent map non-surjective (hence non-injective), from which we can deduce $\cV_{I,\lambda_0}^{\chi_0}(U^u, W_p^u)'$ is not \'etale over $\widehat{Z_{L_I}^o}$ at $z$ for any $I\subsetneq \Delta$; however, if there exists $I\subsetneq \Delta$ such that $w(z)\in \sW_I$, then $\cV_{I,\lambda_0}^{\chi_0}(U^u, W_p^u)$ should be \'etale over $\widehat{Z_{L_I}^o}$ at $z$.
\end{remark}

\section{Local-global compatibility}
Let $\rho_p$ be an $n$-dimensional very regular crystalline representation of $\Gal(\overline{\Q_p}/\Q_p)$ over $E$. Let $\ul{h}=(h_1,\cdots, h_n)$ be the Hodge-Tate weights of $\rho_p$ (with $h_1<h_2\cdots <h_n$), $(\phi_1,\phi_2,\cdots, \phi_n)$ be the eigenvalues of the crystalline Frobenius on $D_{\cris}(\rho_p)$.
Recall $\rho_p$ admits $n!$ triangulations (which are also called refinements) parameterized by $\sW\cong S_n$. Indeed, for $w\in \sW$, one has a triangulation of $\rho_p$ of parameter
\begin{equation*}
\big(\unr\big(\phi_{w^{-1}(1)}\big), \cdots, \unr\big(\phi_{w^{-1}(i)}\big)x^{1-i}, \cdots, \unr\big(\phi_{w^{-1}(n)}\big)x^{1-n}\big)\chi_{w^{\alg}(w)\cdot \lambda}
\end{equation*}
for some $w^{\alg}(w)\in \sW$ (uniquely determined by $w$ and $\rho_p$), where $\lambda$ is the dominant weight of $\ft$ with $k_{\lambda,i}=-h_i+i-1$ for $i=1,\cdots, n$. Let $\psi_{w,i}:=\unr\big(\phi_{w^{-1}(i)}p^{i-n}\big)$, $\psi_w:=\psi_{w,1}\otimes\cdots \psi_{w,n}$. Note by the assumption on $\{\phi_i\}$, the smooth representations $\pi:=\big(\Ind_{\overline{B}(\Q_p)}^{\GL_n(\Q_p)} \psi_{w}\delta_B^{-1}\big)^{\infty}$ of $\GL_n(\Q_p)$ are irreducible and isomorphic to each other \big(recall $\delta_B=\unr(p^{1-n}) \otimes\cdots \otimes \unr(p^{2i-n-1}) \otimes\cdots \otimes \unr(p^{n-1})$\big).
Following Breuil (\cite[\S 6]{Br13I}, see also \cite[\S 6]{Br13II}), for $(w^{\alg}, w)\in S_n \times S_n$, put (note the notation is slight different from that in \cite[\S 6]{Br13II}, indeed, the dominant weight $\lambda$ that we use differs from ``$\lambda_{\tilde{v}}$" in \emph{loc. cit.} by $n-1$)
\begin{equation}\label{equ: gln-ngaa}
  C(w^{\alg},w):=\cF_{\overline{B}}^{\GL_n}\big(\sL(w^{\alg}\cdot (-\lambda)), \psi_w\delta_B^{-1}\big)
\end{equation}
which is irreducible by Thm.\ref{thm: gln-pst} (4).

Let $\rho$ be an $n$-dimensional  continuous representation of  $\Gal(\overline{F}/F)$ over $E$ as in the beginning of \S \ref{sec: 3.3}. Suppose $\rho_{\tilde{v}}$ is crystalline and very regular for all $\tilde{v}|p$. Let $\ul{h}_{\tilde{u}}:=(h_{\tilde{u},1},\cdots, h_{\tilde{u},n})$ denote the Hodge-Tate weights of $\rho_{\tilde{u}}$, $(\phi_{\tilde{u},1}, \cdots, \phi_{\tilde{u},n})$ eigenvalues of crystalline Frobenius on $D_{\cris}(\rho_{\tilde{u}})$ . Let $\lambda_{\tilde{u}}$ be the dominant weight of $\ft$ with $k_{\lambda_{\tilde{u}},i}=-h_{\tilde{u},i}+i-1$ for $i=1,\cdots, n$. We associate as above to $\rho_{\widetilde{u}}$ locally analytic representations $\big\{C(w_{\tilde{u}}^{\alg},w_{\tilde{u}})\big\}_{(w_{\tilde{u}}^{\alg},w_{\tilde{u}})\in S_n\times S_n}$. Suppose moreover $\widehat{\Pi}(\rho)_{\lalg}\neq 0$.
\begin{conjecture}[$\text{\cite{Br13I}, \cite[Conj.5.3]{Br13II}}$]\label{conj: gln-noa}
  Keep the above notation and assumption, then
  \begin{equation*}
    \Hom_{\GL_n(\Q_p)}\big(C(w_{\tilde{u}}^{\alg},w_{\tilde{u}}), \widehat{\Pi}(\rho)_{\an}\big) \neq 0
  \end{equation*}
  if and only if $w_{\tilde{u}}^{\alg}\leq w_{\tilde{u}}^{\alg}(w_{\tilde{u}})$, where the $\GL_n(\Q_p)$ acts on $\widehat{\Pi}(\rho)$ via $i_{G, \tilde{u}}$.
\end{conjecture}
\begin{remark}
  By \cite[Prop.5.4]{Br13II}, the conjecture is in fact independent of the choice of $\tilde{u}$, i.e. if Conj.\ref{conj: gln-noa} holds for $\tilde{u}$ then it holds for $\tilde{u}^c$.
\end{remark}
Recall
\begin{theorem}[$\text{\cite{Br13II}}$]\label{thm: gln-esj} Keep the notation and assumption as in Conj.\ref{conj: gln-noa}.

(1) If $\Hom_{\GL_n(\Q_p)}\big(C(w_{\tilde{u}}^{\alg},w_{\tilde{u}}), \widehat{\Pi}(\rho)_{\an}\big)\neq 0$, then $w_{\tilde{u}}^{\alg}(\ul{h}_{\tilde{u}})\leq w_{\tilde{u}}^{\alg}(w_{\tilde{u}})(\ul{h}_{\tilde{u}})$. If moreover $n<3$ or $\lg(w_{\tilde{u}}^{\alg}(w_{\tilde{u}}))\leq 2$, then $w_{\tilde{u}}^{\alg}\leq w_{\tilde{u}}^{\alg}(w_{\tilde{u}})$.

(2) Suppose $n\leq 3$, or $F/F^+$ unramified, $G$ quasi-split at all finite places and $U_v$ maximal hyperspecial at all inert places, for $w_{\tilde{u}}\in S_n$, if $w^{\alg}_{\tilde{u}}(w_{\tilde{u}})\neq 1$, then there exists $w^{\alg}_{\tilde{u}}\in S_n\setminus \{1\}$, such that
\begin{equation*}
 \Hom_{\GL_n(\Q_p)}\big(C(w_{\tilde{u}}^{\alg},w_{\tilde{u}}), \widehat{\Pi}(\rho)_{\an}\big)\neq 0.
\end{equation*}
\end{theorem}
In particular, when $\lg(w^{\alg}_{\tilde{u}}(w_{\tilde{u}})\leq 1$, Conj.\ref{conj: gln-noa} was proved (under the global hypothesis as in Thm.\ref{thm: gln-esj}).
The following theorem is the main result of this note, which improves Thm.\ref{thm: gln-esj} (2)
\begin{theorem}\label{thm: gln-3dG}Keep the notation and assumption as in Conj.\ref{conj: gln-noa}.
 Suppose $n\leq 3$, or $F/F^+$ unramified, $G$ quasi-split at all finite places and $U_v$ maximal hyperspecial at all inert places. Let $I\subsetneq \Delta$, for $w_{\tilde{u}}\in S_n$, if $w^{\alg}_{\tilde{u}}(w_{\tilde{u}})\notin \sW_I$, then there exists $w^{\alg}_{\tilde{u}}\in S_n$ satisfying
\begin{itemize}
  \item $w^{\alg}_{\tilde{u}}\neq 1$,
  \item $\sL(w^{\alg}_{\tilde{u}}\cdot (-\lambda_{\tilde{u}}))$ is an irreducible component of the generalized Verma module $\text{U}(\ug)\otimes_{\text{U}(\overline{\fp}_I)} \sL_I(-\lambda_{\tilde{u}})$ \big(which implies in particular $w^{\alg}_{\tilde{u}}\cdot (-\lambda_{\tilde{u}})$ is $P_I$-dominant\big),
\end{itemize}
such that
\begin{equation*}
 \Hom_{\GL_n(\Q_p)}\big(C(w_{\tilde{u}}^{\alg},w_{\tilde{u}}), \widehat{\Pi}(\rho)_{\an}\big)\neq 0.
\end{equation*}
\end{theorem} \begin{proof}
  Since $\widehat{\Pi}(\rho)_{\lalg}\neq 0$, we associate to $\rho$ a system of Hecke eigenvalues $\fh_{\rho}:\cH^{S,u}\ra E$. Indeed, by Prop.\ref{prop: gln-stm} and (\ref{equ: gln-ohr}), there exists an automorphic representation $\pi=\pi^{\infty}\otimes \pi_{\infty}$, with $\cH^{S,p}_0$ acting on $(\pi^{\infty})^{U^u}$ by $\cH^{S,p}_0/\fm_{\rho}$ (since $\pi^{\infty}$ is defined over a number field, by enlarging $E$, we assume $\pi^{\infty}$ is defined over $E$). Thus $\cH^{S,u}$ acts on $(\pi^{\infty})^{U^u}$ via a morphism $\fh_{\rho}: \cH^{S,u}\ra E$. Let $\lambda_0:=\lambda_{\tilde{u}}|_{\ft^{\cD}}$, for each $w\in S_n$, we get an $E$-point $z_{w_{\tilde{u}}}=(\chi_{w_{\tilde{u}}},\fh_{\rho})$ in $\cV:=\cV_{I,\lambda_0}^{1}(U^u, W_p^u)$ with $\chi_{w_{\tilde{u}}}=\psi_{w_{\tilde{u}}}\chi_{\lambda_{\tilde{u}}}$ (where ``$1$" denotes the trivial character). Note, with the notation in Thm.\ref{thm: gln-tvI0}, $w^{\alg}_{\tilde{u}}(w_{\tilde{u}})$ is just $w(z_{w_{\tilde{u}}})$. By Thm.\ref{thm: gln-tvI0}, if $w^{\alg}_{\tilde{u}}(w_{\tilde{u}})\notin \sW_I$, then $\cV$ is not \'etale over $\widehat{Z_{L_I}^o}$. Then by Prop.\ref{prop: gln-hlt}, the natural injection
 \begin{multline*}
    J_{B,(P_I,\lambda_0)}\big(\widehat{S}(U^u,W_p^u)_{\lalg}\big)^{T^o=\chi_{w_{\tilde{u}}}}[T(\Q_p)=\chi_{w_{\tilde{u}}},\cH^{S,u}=\fh_{\rho}] \\ \hooklongrightarrow J_{B,(P_I,\lambda_0)}\big(\widehat{S}(U^u,W_p^u)_{\an}\big)^{T^o=\chi_{w_{\tilde{u}}}}[T(\Q_p)=\chi_{w_{\tilde{u}}},\cH^{S,u}=\fh_{\rho}],
  \end{multline*}
  is not surjective. The theorem follows by applying the adjunction formula Thm.\ref{thm: gln-ycn} to the $T(\Q_p)$-representation $J_{B,(P_I,\lambda_0)}\big(\widehat{S}(U^u,W_p^u)_{\an}\big)^{T^o=\chi_{w_{\tilde{u}}}}[T(\Q_p)=\chi_{w_{\tilde{u}}},\cH^{S,u}=\fh_{\rho}]$.
 \end{proof}
In particular, when $n=3$, taking $P_I$ to be a maximal proper parabolic subgroup of $\GL_3$, then ``$w^{\alg}_{\tilde{u}}$" in the theorem is equal to the simple reflection $s\notin \sW_I$ if exists. Indeed, in this case, one has an exact sequence (cf. \cite[\S 9.5]{Hum08})
\begin{equation}\label{equ: gln-gspi}
  0 \ra \sL(s \cdot (-\lambda_{\tilde{u}})) \ra \text{U}(\ug)\otimes_{\text{U}(\overline{\fp}_I)} \sL_I(-\lambda_{\tilde{u}}) \ra \sL(-\lambda_{\tilde{u}})\ra  0.
\end{equation}Thus putting Thm.\ref{thm: gln-esj} (1) and Thm.\ref{thm: gln-3dG} together, one gets
\begin{corollary}\label{cor: gln-ang}
  Suppose $n=3$, let $\alpha\in \Delta$, then $s_{\alpha}\leq w^{\alg}_{\tilde{u}}(w_{\tilde{u}})$ if and only if $C(s_{\alpha}, w_{\tilde{u}})$ is a subrepresentation of $\widehat{\Pi}(\rho)_{\an}$. In particular, if $\lg(w_{\tilde{u}}^{\alg}(w_{\tilde{u}}))\geq 2$, then $\oplus_{\alpha\in \Delta}C(s_{\alpha},w_{\tilde{u}})$ is a subrepresentation of $\widehat{\Pi}(\rho)_{\an}$.
\end{corollary}
In the general setting, let $\cP$ denote the set of maximal proper parabolic subgroups $P_I$ of $\GL_n$ such that $w^{\alg}_{\tilde{u}}(w_{\tilde{u}}) \notin \sW_I$, thus $|\cP|=|\{\alpha\in \Delta\ |\ s_{\alpha}\leq w^{\alg}_{\tilde{u}}(w_{\tilde{u}})\}|$. For each $P_I\in \cP$, by Thm.\ref{thm: gln-3dG} one gets a subrepresentation $C(w^{\alg}_{\tilde{u},I}, w)$ of $\widehat{\Pi}(\rho)_{\an}$ \big(where $w^{\alg}_{\tilde{u},I}$ is the ``$w^{\alg}_{\tilde{u}}$" in Thm.\ref{thm: gln-3dG} applied to $P_I$\big). Note that these $w^{\alg}_{\tilde{u},I}$ are distinct since for $w\in \sW$ if $w\cdot (-\lambda_{\tilde{u}})$ is dominant for two different maximal proper parabolic subgroups, then $w\cdot (-\lambda_{\tilde{u}})$ is dominant, and hence $w=1$. By  \cite[Lem.6.2]{Br13I}, the locally analytic representations $C(w^{\alg}_{\tilde{u},I},w_{\tilde{u}})$ are distinct for different $P_I$. Thus one has
\begin{corollary}\label{cor: gln-nmf}
  Keep the above notation and the assumption in Thm.\ref{thm: gln-3dG}, $\oplus_{P_I\in \cP} C(w^{\alg}_{\tilde{u},I}, w_{\tilde{u}})$ is a subrepresentation of $\widehat{\Pi}(\rho)_{\an}$.
\end{corollary}
Indeed, when $n\geq 4$, for a maximal proper parabolic subgroup $P$ of $\GL_n$, the generalized Verma module $\text{U}(\ug) \otimes_{\text{U}(\overline{\fp})} \sL_I(-\lambda_{\tilde{u}})$ might be more complicated, consequently,  in general Thm.\ref{thm: gln-3dG} could not explicate the ``$w^{\alg}_{\tilde{u},I}$" in Cor.\ref{cor: gln-nmf} (unlike the $\GL_3(\Q_p)$ case as in Cor.\ref{cor: gln-ang}). We end this note by an example for $\GL_4(\Q_p)$.
\begin{example}
  Suppose $n=4$, $\lambda_{\tilde{u}}=0$, we identify the set of simple roots $\Delta $ with $\{1,2,3\}$ as in \S \ref{sec: gln-3.1}, thus  $L_{\{1,2\}}= \GL_3\times \GL_1$, $L_{\{2,3\}}= \GL_1\times \GL_3$ and $L_{\{1,3\}}=\GL_2 \times \GL_2$. Denote by $s_i\in S_4$, $i\in \{1,2,3\}$ the corresponding simple reflection, let $I\subseteq \{1,2,3\}$, $|I|=2$, $i^I\in \{1,2,3\}$, $i^I\notin I$, and denote by $\cS_I$ the set of irreducible composants of the generalized Verma module $\text{U}(\ug) \otimes_{\text{U}(\overline{\fp}_I)} \sL_I(0)$ (which all have multiplicity one by \cite[Thm.8.4]{EnSh}). By \emph{loc. cit.}, one has
  \begin{equation*}
    \cS_I=\begin{cases}
       \big\{\sL(0), \sL(s_{i^I}\cdot 0)\big\} & I\in \{\{1,2\}, \{2,3\}\}\\
       \big\{\sL(0), \sL(s_2\cdot 0), \sL((s_2s_3s_1s_2)\cdot 0)\big\} & I=\{1,3\}
    \end{cases}.
  \end{equation*}
  Thus as in $\GL_3(\Q_p)$ case, one has
  \begin{itemize}
    \item let $i=1$ or $3$, then $s_i\leq w^{\alg}(w_{\tilde{u}})$ if and only if $C(s_i,w_{\tilde{u}})$ is a subrepresentation of $\widehat{\Pi}(\rho)_{\an}$.
  \end{itemize}
  Since $\lambda_{\tilde{u}}=0$, $\ul{h}_{\tilde{u}}=(0,1,2,3)$. One can check for $w\in S_4$, $w(\ul{h}_{\tilde{u}})\geq (s_2s_3s_1s_2)(\ul{h}_{\tilde{u}})$ if and only if $w\geq s_2s_3s_1s_2$. Thus by Thm.\ref{thm: gln-3dG} and Thm.\ref{thm: gln-esj} (1), one gets
   \begin{itemize}\item if $w^{\alg}(w_{\tilde{u}})\ngeq s_2s_3s_1s_2$, then $s_i \leq w^{\alg}(w_{\tilde{u}})$ if and only if $C(s_i, w_{\tilde{u}})$ is a subrepresentation of $\widehat{\Pi}(\rho)_{\an}$ for $i\in \{1,2,3\}$.
   \end{itemize}
   However, if $w^{\alg}(w_{\tilde{u}})\geq s_2s_3s_1s_2$, the author does not know how to see the (conjectured) injection $C(s_2,w_{\tilde{u}})\hookrightarrow \widehat{\Pi}(\rho)_{\an}$.
\end{example}

\end{document}